\documentclass[12pt]{article}
\title{Counting mod $n$ in pseudofinite fields}
\author{Will Johnson}

\usepackage{amsmath, amssymb, amsthm}    	
\usepackage{fullpage} 	
\usepackage{amscd}
\usepackage{hyperref}
\usepackage[all]{xy}
\usepackage{centernot}
\usepackage{titlefoot}

\DeclareMathOperator*{\forkindep}{\raise0.2ex\hbox{\ooalign{\hidewidth$\vert$\hidewidth\cr\raise-0.9ex\hbox{$\smile$}}}}

\newcommand{\Def}{\operatorname{Def}}

\newcommand{\Abs}{\operatorname{Abs}}

\newcommand{\Gal}{\operatorname{Gal}}

\newcommand{\Frac}{\operatorname{Frac}}

\newcommand{\Sym}{\operatorname{Sym}}

\newcommand{\characteristic}{\operatorname{char}}
\newcommand{\Spec}{\operatorname{Spec}}

\newcommand{\Aut}{\operatorname{Aut}}

\newcommand{\acl}{\operatorname{acl}}
\newcommand{\dcl}{\operatorname{dcl}}
\newcommand{\tp}{\operatorname{tp}}

\newcommand{\Fr}{\operatorname{Fr}}

\newcommand{\ACPF}{\mathrm{ACPF}}
\newcommand{\TACPF}{\widetilde{\ACPF}}
\newcommand{\Tr}{\operatorname{Tr}}

\newtheorem{theorem}{Theorem}[section] 
\newtheorem{lemma}[theorem]{Lemma}

\newtheorem{corollary}[theorem]{Corollary}
\newtheorem{fact}[theorem]{Fact}

\newtheorem{conjecture}[theorem]{Conjecture}

\newtheorem{proposition}[theorem]{Proposition}
\newtheorem{proposition-eh}[theorem]{Proposition(?)}
\newtheorem*{theorem-star}{Theorem}
\newtheorem*{conjecture-star}{Conjecture}
\newtheorem*{lemma-star}{Lemma}

\theoremstyle{definition}
\newtheorem{definition}[theorem]{Definition}
\newtheorem{example}[theorem]{Example}

\theoremstyle{remark}
\newtheorem{remark}[theorem]{Remark}
\newtheorem{claim}[theorem]{Claim}

\newcommand{\Aa}{\mathbb{A}}
\newcommand{\Qq}{\mathbb{Q}}

\newcommand{\Rr}{\mathbb{R}}

\newcommand{\Zz}{\mathbb{Z}}
\newcommand{\Nn}{\mathbb{N}}

\newcommand{\Gg}{\mathbb{G}}

\newcommand{\Mm}{\mathbb{M}}
\newcommand{\Ff}{\mathbb{F}}
\newcommand{\Pp}{\mathbb{P}}

\newcommand{\Oo}{\mathcal{O}}

\newenvironment{claimproof}[1][\proofname]
               {
                 \proof[#1]
                 
               }
               {
                 \endproof
               }

\begin{document}
\maketitle\unmarkedfntext{
  \emph{2010 Mathematical Subject Classification}: 03C60, 03C80, 03H05

  \emph{Key words and phrases}: pseudofinite fields, nonstandard
  analysis, Euler characteristics

  The author would like to thank Tianyi Xu, for helpful discussions
  about recursive ind-definability, and Tom Scanlon, who read an
  earlier version of this paper appearing in the author's
  dissertation.

  This material is based upon work supported by the National Science
  Foundation under Grant No. DGE-1106400 and Award No. DMS-1803120.
  Any opinions, findings, and conclusions or recommendations expressed
  in this material are those of the author and do not necessarily
  reflect the views of the National Science Foundation.
}

\begin{abstract}
  We show that in an ultraproduct of finite fields, the mod-$n$
  nonstandard size of definable sets varies definably in families.
  Moreover, if $K$ is any pseudofinite field, then one can assign
  ``nonstandard sizes mod $n$'' to definable sets in $K$.  As $n$ varies, these
  nonstandard sizes assemble into a definable strong Euler
  characteristic on $K$, taking values in the profinite completion
  $\hat{\Zz}$ of the integers.  The strong Euler characteristic is not
  canonical, but depends on the choice of a nonstandard Frobenius.
  When $\Abs(K)$ is finite, the Euler characteristic has some funny
  properties for two choices of the nonstandard Frobenius.

  Additionally, we show that the theory of finite fields remains
  decidable when first-order logic is expanded with parity
  quantifiers.  However, the proof depends on a computational
  algebraic geometry statement whose proof is deferred to a later
  paper.
\end{abstract}

\section{Introduction}
\subsection{Euler characteristics}
Let $M$ be a structure and $R$ be a ring.  Let $\Def(M)$ denote the
collection of (parametrically) definable sets in $M$.  Recall the
following definitions from \cite{krajicek} and \cite{euler-papers}.
An \emph{$R$-valued Euler characteristic} is a function $\chi :
\Def(M) \to R$ such that
\begin{itemize}
\item $\chi(\emptyset) = 0$
\item $\chi(X) = 1$ if $X$ is a singleton
\item $\chi(X) = \chi(Y)$ if $X$ and $Y$ are in definable bijection.
\item $\chi(X \times Y) = \chi(X) \cdot \chi(Y)$
\item $\chi(X \cup Y) = \chi(X) + \chi(Y)$ if $X$ and $Y$ are disjoint.
\end{itemize}
If the following additional property holds, then $\chi$ is called a
\emph{strong} Euler characteristic:
\begin{itemize}
\item
  If $f : X \to Y$ is a definable function and there is an $r \in R$
  such that $\chi(f^{-1}(y)) = r$ for all $y$, then
  \[ \chi(X) = r \cdot \chi(Y).\]
\end{itemize}
An Euler characteristic $\chi$ is \emph{definable} if the set $\{y \in
Y : \chi(f^{-1}(y)) = r\}$ is definable for every definable function
$f : X \to Y$ and every $r \in R$.

\subsection{Examples of Euler characteristics}
The simplest example of an Euler characteristic is the counting
function on a finite structure.  If $M$ is a finite structure, there
is a $\Zz$-valued Euler characteristic given by
\begin{equation*}
  \chi(X) = |X|
\end{equation*}
where $|X|$ denotes the size of $X$.  This $\chi$ is always strong and
0-definable.

Another well-known example is the Euler characteristic on dense
o-minimal structures (\cite{lou-o-minimality}, \S 4.2).  If
$(M,<,\ldots)$ is a dense o-minimal structure, there is a $\Zz$-valued
Euler characteristic on $M$, characterized by the fact that $\chi(C) =
-1^{\dim C}$ for any open cell $C$.  This Euler characteristic is
strong and 0-definable.  By work of Kamenkovich and Peterzil
\cite{kamenkovich}, it can be extended to $M^{eq}$.  On o-minimal
expansions of the reals, $\chi(X)$ agrees with the topological Euler
characteristic for compact definable $X \subseteq \Rr^n$.

Pseudofinite structures have strong Euler characteristics arising from
counting mod $n$.  More precisely, if $M$ is an ultraproduct of finite
structures, there is a canonical strong Euler characteristic $\chi_n :
\Def(M) \to \Zz/n\Zz$ defined in the following way.  Let $M$ be the
ultraproduct $\prod_{i \in I} M_i/\mathcal{U}$, and $X = \phi(M;a)$
be a definable set.  Choose a tuple $\langle a_i \rangle_{i \in I} \in
\prod_{i \in I} M_i$ representing $a$.  Then define $\chi_n(X) \in
\Zz/n\Zz$ to be the ultralimit along $\mathcal{U}$ of the sequence
\begin{equation*}
  \langle |\phi(M_i;a_i)| + n\Zz \rangle_{i \in I}
\end{equation*}
This ultralimit exists because $\Zz/n\Zz$ is finite.

More intuitively, if we take $\Zz^*$ to be the ultrapower
$\Zz^{\mathcal{U}} \succeq \Zz$, then there is a nonstandard counting
function $\chi^*: \Def(M) \to \Zz^*$ assigning to each definable set
$X \subseteq M^n$ its nonstandard ``size'' in $\Zz^*$.  Then $\chi_n$
is the composition
\begin{equation*}
  \Def(M) \stackrel{\chi^*}{\to} \Zz^* \rightarrow \Zz^*/n\Zz^*
  \xrightarrow{\sim} \Zz/n\Zz
\end{equation*}
The map $\chi^*$ happens to be a strong Euler characteristic itself,
but we will not consider it further.

The mod $n$ Euler characteristics on pseudofinite structures need not
be definable.  For example, consider an ultraproduct of the
totally ordered sets $\{0,1,\ldots,m\}$ as $m \to \infty$.  The resulting ultraproduct $(M,<)$ admits no definable
$\Zz/n\Zz$-valued Euler characteristics (for $n > 1$).  Indeed, if
$\chi$ is an Euler characteristic on $M$, consider the function
\begin{equation*}
  f(a) = \chi\left([0,a]\right) \in \Zz/n\Zz
\end{equation*}
Then $f(b) = f(a)+1$ when $b$ is the successor of $a$.  The set
$f^{-1}(0)$ must therefore contain every $n$th element of $M$, and
hence cannot be definable, because $M$ is (non-dense) o-minimal.

We will see below (Theorem~\ref{psy-thm}.\ref{ultra-psy-copy}) that this
does \emph{not} happen with ultraproducts of finite fields: the
$\chi_n$ are always definable on ultraproducts of finite fields.

On an ultraproduct $M$ of finite structures, these $\chi_n$ maps are
compatible in the sense that the following diagram commutes when $n$
divides $m$:
\begin{equation*}
  \xymatrix{\Def(M) \ar[r]^{\chi_m} \ar[rd]^{\chi_n} & \Zz/m\Zz \ar[d]
    \\ & \Zz/n\Zz}
\end{equation*}
Consequently, they assemble into a map
\begin{equation*}
  \hat{\chi} : \Def(M) \to \hat{\Zz}
\end{equation*}
where $\hat{\Zz}$ is the ring $\varprojlim_{n \in \Nn} \Zz/n\Zz$.

More generally, if $M$ is any structure, we will say that a map $\chi
: \Def(M) \to \hat{\Zz}$ is
\begin{enumerate}
\item an \emph{Euler characteristic} if all the compositions $\Def(M)
  \to \hat{\Zz} \to \Zz/n\Zz$ are Euler characteristics
\item\label{sechar} a \emph{strong} Euler characteristic if all the
  compositions $\Def(M) \to \hat{\Zz} \to \Zz/n\Zz$ are strong Euler
  characteristics
\item\label{dechar} a \emph{definable} Euler characteristic if all the
  compositions $\Def(M) \to \hat{\Zz} \to \Zz/n\Zz$ are definable
  Euler characteristics.
\end{enumerate}
For \ref{sechar} and \ref{dechar}, this is an abuse of terminology.

We can repeat the discusison above with the $p$-adics $\Zz_p =
\varprojlim_{k} \Zz/p^k\Zz$ instead of $\hat{\Zz}$.  Recall that
\begin{equation*}
  \hat{\Zz} \cong \prod_p \Zz_p
\end{equation*}
by the Chinese remainder theorem.  Giving an Euler characteristic
$\hat{\chi}:\Def(M) \to \hat{\Zz}$ is therefore equivalent to giving
an Euler characteristic $\chi_p : \Def(M) \to \Zz_p$ for every $p$.
Moreover, $\hat{\chi}$ is strong or definable if and only if every
$\chi_p$ is strong or definable, respectively.  It is sometimes more
convenient to work with $\Zz_p$ because it is an integral domain,
unlike $\hat{\Zz}$.

\subsection{Main results for pseudofinite fields} \label{sec-mrfields}
A structure is \emph{pseudofinite} if it is infinite, yet
elementarily equivalent to an ultraproduct of finite structures.  By a
theorem of Ax \cite{ax-paper}, a field $K$ is pseudofinite if and only
if $K$ satsifies the following three conditions:
\begin{itemize}
\item $K$ is perfect
\item $K$ is pseudo-algebraically closed: every geometrically integral
  variety over $K$ has a $K$-point.
\item $\Gal(K) \cong \hat{\Zz}$, or equivalently, $K$ has a unique
  field extension of degree $n$ for each $n$.
\end{itemize}

Our first main result can be phrased purely in terms of pseudofinite
fields.
\begin{theorem}\label{psy-thm}
  \hfill
  \begin{enumerate}
  \item Let $K = \prod_i K_i/\mathcal{U}$ be an ultraproduct of
    finite fields.  Then the nonstandard counting functions $\chi_n$
    are $\acl^{eq}(\emptyset)$-definable.\label{ultra-psy-intro}
  \item Every pseudofinite field admits an
    $\acl^{eq}(\emptyset)$-definable $\hat{\Zz}$-valued strong Euler
    characteristic.\label{rando-psy-intro}
  \end{enumerate}
\end{theorem}
We make several remarks:
\begin{enumerate}
\item In Part~\ref{ultra-psy-intro}, the $\acl^{eq}(\emptyset)$ is
  necessary: the nonstandard counting function is known to not be
  0-definable, by Theorem~7.3 in \cite{krajicek}.
\item In Part~\ref{rando-psy-intro}, the Euler characteristic is not
  canonical, but depends on a choice of a topological generator
  $\sigma \in \Gal(K)$.
\end{enumerate}
One approach to proving Theorem~\ref{psy-thm} would be to use etale
cohomology.  In fact, there should be a close connection between
$\ell$-adic cohomology and the $\ell$-adic part of the Euler
characteristic $\chi$---see Conjecture~\ref{weilish}.  This approach
was originally suggested by Hrushovski, according to Kraj\'i\v{c}ek's
comments at the end of \cite{krajicek}.

We avoid this line of proof, because it is less elementary, and
doesn't handle the case where $\ell = \characteristic(K)$.  Rather
than using etale cohomology, we will use the more elementary theory of
abelian varieties and jacobians, essentially falling back to Weil's
original proof of the Riemann hypothesis for curves.

Aside from Theorem~\ref{psy-thm}, there is also a decidability theorem
in terms of generalized parity quantifiers.  For any $n \in \Nn$ and
$k \in \Zz/n\Zz$, let $\mu^n_k x$ be a new quantifier.  Interpret
$\mu^n_k x : \phi(x)$ in finite structures as
\begin{quote}
  The number of $x$ such that $\phi(x)$ holds is congruent to $k$ mod $n$.
\end{quote}
In other words,
\begin{equation*}
  \left(M \models \mu^n_k \vec{x} : \phi(\vec{x},\vec{b})\right) \iff \left(|\{\vec{a} :
  M \models \phi(\vec{a},\vec{b})\}| \equiv k \pmod{n}\right).
\end{equation*}
For example,
\begin{itemize}
\item $\mu^2_0 x$ means ``there are an even number of $x$ such that\ldots''
\item $\mu^2_1 x$ means ``there are an odd number of $x$ such that\ldots''
\end{itemize}
We call $\mu^n_k$ a \emph{generalized parity quantifier}.

Let $\mathcal{L}^{\mu}_{rings}$ be the language of rings expanded with
generalized parity quantifiers.
\begin{theorem}\label{intro-deci}
  Assuming Conjecture~\ref{horror}, the
  $\mathcal{L}^\mu_{rings}$-theory of finite fields is decidable.
\end{theorem}
Unfortunately, this result is conditional on Conjecture~\ref{horror},
a technical statement about definability in algebraic geometry.  While
the conjecture is certainly true, it is hard to give a sane proof, for
reasons discussed in \S\ref{sec:interlude}.  A complete proof will
(hopefully) be given in future work \cite{complementary-paper}.

\subsection{Main results for periodic difference fields}
The results of \S\ref{sec-mrfields} can be stated more precisely in
terms of difference fields.  Recall that a \emph{difference field} is
a pair $(K,\sigma)$ where $K$ is a field and $\sigma$ is an
automorphism of $K$.
\begin{definition}
  A \emph{periodic} difference field is a difference field
  $(K,\sigma)$ such that every element of $K$ has finite orbit under
  $\sigma$.
\end{definition}
Periodic difference fields are not an elementary class in the language
of difference fields.  However, they constitute an elementary class
when regarded as multi-sorted structures $(K_1,K_2,\ldots)$ where
$K_i$ is the fixed field of $\sigma^i$, with the following structure:
\begin{itemize}
\item The difference-field structure on each $K_i$
\item The inclusion map $K_n \to K_m$ for each pair $n, m$ with $n$
  dividing $m$
\end{itemize}
These multi-sorted structures were considered by Hrushovski in
\cite{the-lost-paper}, and we will give an overview of their basic
properties in \S\ref{pf-review} below.

To highlight the fact that we are no longer working in the language of
difference fields, we will call these structures \emph{periodic
  fields}.  If $(K_1,K_2,\ldots)$ is a periodic field, we let
$K_\infty$ denote the associated periodic difference field
\[ K_\infty = \varinjlim_n K_n.\]
We will abuse notation and write $(K_\infty,\sigma)$ when we really
mean the associated periodic field $(K_1,K_2,\ldots)$.

For any $q$, let $\Fr^q_\infty$ denote $(\Ff_q^{alg},\phi_q)$, where $\phi_q$
is the $q$th power Frobenius.  Thus $\Fr^q_n = (\Ff_{q^n},\phi_q)$.  We will call the $\Fr^q_\infty$'s \emph{Frobenius
  periodic fields}.  Frobenius periodic fields are essentially finite,
in the sense that every definable set is finite.  Consequently,
ultraproducts of Frobenius periodic fields admit $\Zz/n\Zz$-valued
strong Euler characteristics $\chi_n$.

There is a theory ACPF whose class of models can be described in
several ways:
\begin{enumerate}
\item The existentially closed periodic fields.
\item The non-Frobenius periodic fields satisfying the theory of
  Frobenius periodic fields.
\item The periodic fields of the form $(K^{alg},\sigma)$, where $K$ is
  pseudofinite and $\sigma$ is a topological generator of $\Gal(K)$.
\end{enumerate}
(See Propositions~\ref{acpf-char}, \ref{acpf-is-ultraprod}, and
\ref{acpf-is-pseuf}, respectively.)  In particular, ACPF is the model
companion of periodic fields, and non-principal ultraproducts of
Frobenius periodic fields are models of ACPF.  The situation is
analogous to, but much simpler than, the situation with ACFA
\cite{frobenius}.

Theorem~\ref{psy-thm} has the following analogue for periodic fields:
\begin{theorem}\label{ecpdf-thm}
  Let $\mathcal{C}$ be the class of Frobenius periodic fields and
  existentially closed periodic fields.  There is a $\hat{\Zz}$-valued
  strong Euler characteristic $\chi$ on $(K,\sigma)$ in $\mathcal{C}$ with the
  following properties:
  \begin{itemize}
  \item $\chi$ is uniformly 0-definable across $\mathcal{C}$.
  \item If $(K,\sigma)$ is a Frobenius periodic field, then $\chi$ is
    the counting Euler characteristic:
    \[ \chi(X) = |X|.\]
  \item If $(K,\sigma)$ is an ultraproduct of Frobenius periodic
    fields, then $\chi$ is the nonstandard counting Euler
    characteristic.
  \end{itemize}
\end{theorem}
If $F$ is an abstract pseudofinite field, each topological generator
$\sigma \in \Gal(F)$ turns $(F^{alg},\sigma)$ into a periodic
difference field satisfying ACPF.  There is no canonical choice of
$\sigma$, which is the reason for the non-canonicalness in
Theorem~\ref{psy-thm}.\ref{rando-psy-copy}.

There are also statements in terms of parity quantifiers.  Let
$\mathcal{L}_{pf}$ be the first-order language of periodic fields, and
let $\mathcal{L}^{\mu}_{pf}$ be its expansion by generalized parity
quantifiers.
\begin{theorem}\label{intro-deci-2}
  \hfill
  \begin{enumerate}
  \item \label{id2a} Generalized parity quantifiers are uniformly eliminated on the
    class of Frobenius periodic fields.
  \item \label{id2b} Assuming Conjecture~\ref{horror}, the
    $\mathcal{L}^{\mu}_{pf}$-theory of Frobenius periodic fields is
    decidable.
  \end{enumerate}
\end{theorem}
This statement is stronger than what we can say about finite and
pseudofinite fields.  In fact, generalized parity quantifiers are
\emph{not} uniformly eliminated on finite fields (Lemma~\ref{lame}).

\subsection{A special case}\label{pre-mock}
If $p$ is a prime, let $\Zz_{\neg p}$ be the prime-to-$p$ completion
of $\Zz$:
\begin{equation*}
  \Zz_{\neg p} = \varprojlim_{(n,p) = 1} \Zz/n\Zz = \prod_{\ell \ne p}
  \Zz_\ell.
\end{equation*}
If $K$ is a field, let $\Abs(K)$ denote the subfield of \emph{absolute
  numbers}, i.e., the relative algebraic closure of the prime field.
Say that a field $K$ is \emph{mock-finite} if $K$ is pseudofinite and
$\Abs(K)$ is finite.  Say that $K$ is a \emph{mock-$\Ff_q$} if moreover
$\Abs(K) \cong \Ff_q$.  For each prime power $q$, there is a unique
mock-$\Ff_q$ up to elementary equivalence
(Proposition~\ref{prop:frob}.\ref{z3}).

The nonstandard Euler characteristics behave in a funny way on
mock-finite fields:
\begin{theorem}\label{mocktastic}
  Let $K$ be a mock-$\Ff_q$, for some prime power $q = p^k$.  There
  are two $\Zz_{\neg p}$-valued 0-definable strong Euler
  characteristics $\chi$ and $\chi^\dag$ on $K$, such that
  \begin{enumerate}
  \item If $V$ is a smooth projective variety over $\Ff_q$, then
    \begin{align*}
      \chi(V(K)) &= |V(\Ff_q)| \\
      \chi^\dag(V(K)) &= |V(\Ff_q)|/q^{\dim V}.
    \end{align*}
  \item If $X$ is any $\Ff_q$-definable set, then
    \[ \chi(X) = |X \cap \dcl(\Ff_q)|.\]
    In particular, $\chi(X) \in \Zz$.
  \item If $X$ is any $\Ff_q$-definable set, then $\chi^\dag(X) \in \Qq$.
  \end{enumerate}
\end{theorem}
Using this, we construct a strange $\Qq$-valued weak Euler
characteristic on pseudofinite fields in \S\ref{sec-neutral}.

\subsection{Related work}
Many people have considered non-standard sizes of definable sets in
pseudofinite fields \cite{alshanqiti, CDM, elwes, krajicek,
  euler-papers}.  Non-standard sizes modulo $p$ were considered by
Kraj\'i\v{c}ek, who used them to prove the existence of non-trivial
strong Euler characteristics on pseudofinite fields \cite{krajicek}.
However, most research has focused on \emph{ordered} Euler
characteristics (\cite{alshanqiti, euler-papers}) and the \emph{real}
standard part of non-standard sizes (\cite{CDM, elwes}).  These topics
can be seen as ``non-standard sizes modulo the infinite prime.''

Dwork~\cite{dwork} and Kiefe~\cite{catarina} consider the behavior of
$|\phi(\Ff_q)|$ as $q$ varies.  Their work can be used to calculate
the non-standard mod-$n$ sizes of 0-definable sets in pseudofinite
fields of positive characteristic.

Almost everything in \S\ref{pf-review} is well-known to experts.  The
results specific to periodic fields probably appear in Hrushovski's
paper \cite{the-lost-paper}, which I have had trouble finding.

\subsection{Notation}
If $K$ is a field, then $K^{alg}$ (resp. $K^{sep}$) denotes the
algebraic (resp. separable) closure, and $\Gal(K)$ denotes the
absolute Galois group $\Gal(K^{sep}/K) = \Aut(K^{sep}/K)$.  We let
\[ \hat{\Zz} = \varprojlim_n \Zz/n\Zz\]
denote the profinite completion of $\Zz$.  The finite field with $q$
elements is denoted $\Ff_q$.

A \emph{variety} over $K$ is a finite-type separated reduced scheme
over $K$, not necessarily irreducible or quasi-projective.  If $V$ is
a variety, then $V(K)$ denotes the set of $K$-points of $V$.  A scheme
$X$ over $K$ is \emph{geometrically integral} or \emph{geometrically
  irreducible} if $X \times_K K^{alg}$ is integral or irreducible.  A
\emph{curve} over $K$ is a geometrically integral 1-dimensional smooth
projective variety over $K$.
\begin{remark}
  If $K$ is a perfect field and $V$ is a variety, then geometrically
  irreducible is equivalent to geometrically integral.
\end{remark}

\section{Review of abelian varieties}
Let $A$ be an abelian variety over some field $K$.  For any $n \in
\Nn$, let $A[n]$ denote the group of $n$-torsion in $A(K^{alg})$,
viewed as an abelian group with $\Gal(K)$-action.  The $\ell$th
\emph{Tate module} is defined as an inverse limit
\[ T_\ell A = \varprojlim_n A[\ell^n].\]
See \S 18 of \cite{mumford} for a precise definition.  If $g = \dim
A$, then there are non-canonical isomorphisms
\[ T_\ell A \approx \Zz_\ell^{2g}\]
for all $\ell \ne \characteristic(K)$.  In particular, $T_\ell A$ is a
free $\Zz_\ell$-module of rank $2g$.  If $p = \characteristic(K)$,
then
\[ T_p A \approx \Zz_p^r\]
for some $r$ known as the \emph{$p$-rank} of $A$.  The $p$-rank is at
most $g$.  Similar statements hold for the torsion subgroups:
\begin{align*}
  A[\ell^k] &\approx (\Zz/\ell^k)^{2g} \qquad \ell \ne \characteristic(K) \\
  A[p^k] &\approx (\Zz/p^k)^r \qquad p = \characteristic(K).
\end{align*}
An \emph{isogeny} on $A$ is a surjective endomorphism $f : A \to A$.
An isogeny $f$ is finite and flat (\cite{milneAV}, Proposition~I.7.1),
hence has a well-defined degree $\deg(f)$.  Degree of finite flat maps
is preserved in pullbacks, so $\deg(f)$ can be described alternately
as
\begin{itemize}
\item The length of the scheme-theoretic kernel of $f$ (a finite group
  scheme over $K$).
\item The degree of the fraction field extension.
\end{itemize}
If $f : A \to A$ is a non-surjective endomorphism, then $\deg(f)$ is
defined to be 0.

Any endomorphism $f : A \to A$ induces an endomorphism $T_\ell(f)$ on
the Tate modules.  We can talk about the determinant and trace of this
endomorphism.
\begin{fact}[cf. Theorem~19.4 in \cite{mumford}, or Proposition~I.10.20 in \cite{milneAV}]
  \label{t-ell-magic-2}
  If $f : A \to A$ is any endomorphism, and $\ell \ne
  \characteristic(K)$, then
  \begin{equation*}
    \deg(f) = \det T_\ell(f).
  \end{equation*}
\end{fact}
\begin{corollary} \label{poly-roots}
  If $\alpha_1, \ldots, \alpha_{2g}$ denote the eigenvalues of
  $T_\ell(f)$, then for any polynomial $P(X) \in \Zz[X]$,
  \begin{equation*}
    \deg(P(f)) = \prod_{i = 1}^{2g} P(\alpha_i).
  \end{equation*}
  Because the left hand side is an integer independent of $\ell$, it
  follows that the $\alpha_i$ are algebraic numbers which do not
  depend on $\ell$.
\end{corollary}
The numbers $\alpha_1, \ldots, \alpha_{2g}$ are called the
\emph{characteristic roots} of the endomorphism $f$.  The
characteristic roots govern the counting of points on curves over
finite fields:
\begin{fact}[= Theorem~III.11.1 in \cite{milneAV}]\label{curve-fact}
  Let $C$ be a curve over a finite field $\Ff_q$, and let $J$ be its
  Jacobian.  Then
  \begin{equation*}
    |C(\Ff_q)| = 1 - \left(\sum_{i = 1}^{2g} \alpha_i\right) + q
  \end{equation*}
  where the $\alpha_i$ are the characteristic roots of the $q$th power
  Frobenius endomorphism $\phi_q : J \to J$.
\end{fact}
\begin{corollary}\label{trace-count}
  In the setting of Theorem~\ref{curve-fact}, if $\ell$ is prime to $q$, then
  \begin{equation*}
    |C(\Ff_q)| \equiv 1 - \Tr(\phi_q | J[\ell^k]) + \Tr(\phi_q | \Gg_m[\ell^k]) \pmod{\ell^k}
  \end{equation*}
  where $\Gg_m$ denotes the multiplicative group, $\Gg_m[\ell^k]$
  denotes the group of $\ell^k$th roots of unity (in $\Ff_q^{alg}$),
  and $\Tr(\sigma | M)$ denotes the trace of an endomorphism $\sigma$
  of some free $\Zz/\ell^k$-module $M$.
\end{corollary}
\begin{proof}
  First of all note that there are non-canonical isomorphisms
  \begin{align*}
    J[\ell^k] &\approx (\Zz/\ell^k)^{2g} \\
    \Gg_m[\ell^k] &\approx \Zz/\ell^k
  \end{align*}
  and so the modules are indeed free $\Zz/\ell^k$-modules, and the
  traces are meaningful.  The trace $\Tr(\phi_q | J[\ell^k])$ is
  simply the $\ell^k$-residue class of $\Tr(\phi_q | T_\ell J)$.  The
  action of $\phi_q$ on $\Gg_m$ is multiplication by $q$, so
  $\Tr(\phi_q | \Gg_m[\ell^k])$ is exactly $q$ (mod $\ell^k$).
\end{proof}

\subsection{Bad characteristic}
We would like an analogue of Corollary~\ref{trace-count} in the case
of bad characteristic $\ell = p$.

\begin{lemma}\label{submultiset}
  Let $Q(x)$ and $R(x)$ be two monic polynomials in $\Qq_p[x]$.  Let
  $\beta_1, \ldots, \beta_m \in \Qq_p^{alg}$ be the roots of $Q(x)$,
  and $\alpha_1, \ldots, \alpha_n \in \Qq_p^{alg}$ be the roots of
  $R(x)$.  Suppose that
  \begin{equation}
    v_p\left(\prod_{i = 1}^m P(\beta_i)\right) \le v_p \left( \prod_{i
      = 1}^n P(\alpha_i) \right) \label{mess-e}
  \end{equation}
  holds for every $P(x) \in \Zz[x]$.  Then
  $\{\beta_1,\ldots,\beta_m\}$ is a submultiset of
  $\{\alpha_1,\ldots,\alpha_m\}$, i.e., $Q(x)$ divides $R(x)$.
\end{lemma}
\begin{proof}
  Let $\gamma_1$ be any element of $\Qq_p^{alg}$.  Let $q_1$ and $r_1$
  be the multiplicities of $\gamma_1$ as a root of $Q(x)$ and $R(x)$,
  respectively.  (Either can be zero.)  We will show that $q_1 \le
  r_1$.
  
  The identity (\ref{mess-e}) extends by continuity to any $P(x) \in
  \Zz_p[x]$.  Let $\{\gamma_1, \ldots, \gamma_\ell\} \subseteq
  \Qq_p^{alg}$ be the set of conjugates of $\gamma_1$ over $\Qq_p$.
  For some non-zero $a \in \Zz_p$, the polynomial
  \begin{equation*}
    P(x) = a(x - \gamma_1) \cdots (x - \gamma_\ell)
  \end{equation*}
  lies in $\Zz_p[x]$.  For any $\epsilon \in \Zz_p$, we can apply
  (\ref{mess-e}) to $P(x + \epsilon)$, yielding
  \begin{equation*}
    v_p\left(a^m \prod_{i = 1}^m \prod_{j = 1}^\ell (\beta_i + \epsilon - \gamma_j)\right)
    \le
    v_p\left(a^n \prod_{i = 1}^n \prod_{j = 1}^\ell (\alpha_i + \epsilon - \gamma_j)\right),
  \end{equation*}
  or equivalently,
  \begin{equation}
    v_p\left(a^m \prod_{j = 1}^\ell Q(\gamma_j - \epsilon)\right) \le
    v_p\left(a^n \prod_{j = 1}^\ell R(\gamma_j - \epsilon)\right). \label{mess-d}
  \end{equation}
  Let $q_j$ and $r_j$ be the multiplicity of $\gamma_j$ as a root of
  $Q(x)$ and $R(x)$, respectively.  If $v_p(\epsilon) \gg 0$, then
  (\ref{mess-d}) yields
  \begin{equation*}
    O(1) + v_p(\epsilon) \cdot \sum_{j = 1}^\ell q_j \le O(1) + v_p(\epsilon)
    \cdot \sum_{j = 1}^\ell r_j.
  \end{equation*}
  Thus $\sum_{j = 1}^\ell q_j \le \sum_{j = 1}^\ell r_j$.  But in fact
  $q_j = q_1$, because $Q(x)$ is over $\Qq_p$.  Similarly, $r_j = r_1$
  independent of $j$.  Thus $\ell \cdot q_1 \le \ell \cdot r_1$.
\end{proof}
Recall that the degree of an isogeny $f : A \to A$ is equal to the
degree of the fraction field extension, and therefore factors into separable and inseparable parts:
\begin{equation*}
  \deg(f) = \deg_s(f) \cdot \deg_i(f).
\end{equation*}
Moreover, $\deg_s(f)$ is the size of the set-theoretic kernel of $f$
(\cite{mumford}, \S6, Application 3).
\begin{fact}\label{ell-can-be-p}
  For any $\ell$ (possibly $\ell = p$),
  \begin{equation*}
    v_\ell(\det T_\ell(\phi)) = v_\ell(|\ker \phi|) = v_\ell(\deg_s(\phi))
  \end{equation*}
\end{fact}
Fact~\ref{ell-can-be-p} is implicit in the proof of Theorem~19.4 in
\cite{mumford} or Theorem~I.10.20 in \cite{milneAV}.
\begin{lemma}\label{the-other-r}
  Let $A$ be an abelian variety over $\Ff_q$ for $q = p^k$.  Let
  $\beta_1, \ldots, \beta_r$ be the eigenvalues of $T_p(\phi_q)$, for
  $\phi_q$ the $q$th power Frobenius on $A$.
  \begin{enumerate}
  \item \label{p1} $\{\beta_1, \ldots, \beta_r\}$ is a submultiset of
    the characteristic roots $\{\alpha_1,\ldots,\alpha_r\}$ of
    $\phi_q$.
  \item \label{p2} Each $\beta_i$ has valuation zero in $\Qq_p^{alg}$.
  \end{enumerate}
\end{lemma}
\begin{proof}
  By Corollary~\ref{poly-roots} and Fact~\ref{ell-can-be-p}, the
  following holds for any polynomial $P(x) \in \Zz[x]$:
  \begin{align*}
    v_p\left(\prod_{i = 1}^r P(\beta_i)\right) &= v_p(\det
    T_p(P(\phi_q))) = v_p(\deg_s(P(\phi_q))) \\ & \le v_p(\deg(P(\phi_q)))
    = v_p \left(\prod_{i = 1}^{2g} P(\alpha_i)\right).
  \end{align*}
  Then (\ref{p1}) follows by Lemma~\ref{submultiset}.  For (\ref{p2}),
  note that the $\beta_i$ are integral over $\Zz_p$ because they are
  the eigenvalues of a linear map $\Zz_p^r \to \Zz_p^r$.  Integrality
  implies that $v_p(\beta_i) \ge 0$.  Moreover, the map $\Zz_p^r \to
  \Zz_p^r$ is invertible, because the $q$th power Frobenius is a
  bijection on points.  Therefore, the $\beta_i^{-1}$ are also
  integral, of nonnegative valuation.
\end{proof}

\begin{lemma}\label{sparta-1}
  There is a computable function $h_1(d,d',p,s)$ with the following
  property.  Let $(K,v)$ be an algebraically closed valued field of
  mixed characteristic $(0,p)$.  Let $Q(x)$ be a monic polynomial of
  degree $d$, with roots $\alpha_1, \ldots, \alpha_d$.  Suppose $d'
  \le d$ and suppose that $v(Q(p^i)) \ge v(p^{id'})$ for $1 \le i \le
  h_1(d,d',p,s)$.  Then at least $d'$ of the $\alpha_i$ satisfy
  $v(\alpha_i) \ge v(p^s)$.
\end{lemma}
\begin{proof}
  We first claim that $h_1(d,d',p,s)$ exists for fixed $d, d', p, s$.
  Otherwise, by compactness there is $(K,v) \models
  \mathrm{ACVF}_{0,p}$ and a monic polynomial $Q(x)$ of degree $d$
  such that
  \[ \forall i \in \Nn : v(Q(p^i)) \ge v(p^{id'}),\]
  but fewer than $d'$ of the roots of $Q(x)$ have valuation greater
  than $v(p^s)$.  Let $\alpha_1, \ldots, \alpha_d$ be the roots of
  $Q(x)$, sorted so that
  \[ v(\alpha_1) \ge v(\alpha_2) \ge \cdots \ge v(\alpha_d).\]
  Say that $\alpha_j$ is ``infinitesimal'' if $v(\alpha_j) \ge v(p^n)$
  for every $n \in \Nn$.  Then $\alpha_1, \ldots, \alpha_k$ are
  infinitesimal, and $\alpha_{k+1}, \ldots, \alpha_d$ are not, for
  some $k \le d$.  We claim $k \ge d'$.  Otherwise, take $i \in \Nn$ so large that
  \begin{align*}
    v(p^i) & > v(\alpha_{k+1}) \\
    (d' - k) \cdot v(p^i) & > \sum_{j = k+1}^d v(\alpha_j).
  \end{align*}
  Note
  \[ v(\alpha_k) > v(p^i) > v(\alpha_{k+1}).\]
  Then
  \begin{align*}
    v(Q(p^i)) &= \sum_{j = 1}^k v(p^i - \alpha_j) + \sum_{j = k+1}^d v(p^i - \alpha_j) \\
    &= \sum_{j = 1}^k v(p^i) + \sum_{j = k+1}^d v(\alpha_j) \\
    &= k \cdot v(p^i) + \sum_{j = k+1}^d v(\alpha_j).
  \end{align*}
  By assumption, $v(Q(p^i)) \ge v(p^{id'}) = d' \cdot v(p^i)$, and so
  \[ k \cdot v(p^i) + \sum_{j = k+1}^d v(\alpha_j) \ge d' \cdot v(p^i),\]
  contradicting the choice of $i$.

  Therefore $k \ge d'$.  So at least $d'$ of the roots of $Q(x)$ are
  infinitesimal, hence have magnitude greater than or equal to $p^s$,
  a contradiction.  This shows that $h_1(d,d',p,s)$ exits for each $d,
  d', p, s$.  Now if $\tau_{d,d',p,s,h}$ is the first-order sentence
  expressing that $h$ has the desired property with respect to $d, d',
  p, s$, then
  \begin{equation*}
    \forall d, d', p, s ~ \exists h : \mathrm{ACVF} \vdash \tau_{d,d',p,s,h}.
  \end{equation*}
  Because $\tau_{d,d',p,s,h}$ depends computably on $d, d', p, s, h$,
  and the set of theorems in ACVF is computably enumerable, one can
  choose $h$ to depend computably on $d, d', p, s$.
\end{proof}

\begin{lemma}\label{local-artin}
  Let $G$ be a finite connected commutative group scheme of length $n$
  over $\Ff_q$.  If $n < q$ then the $q$th-power Frobenius morphism $G
  \to G$ is the zero endomorphism.
\end{lemma}
\begin{proof}
  We can write $G$ as $\Spec A$ for some local Artinian
  $n$-dimensional $\Ff_q$-algebra $A$.  Let $\mathfrak{m}$ be the
  maximal ideal of $A$; by properties of local Artinian rings this is
  the sole prime ideal.  We claim that the $\Ff_q$-algebra
  $A/\mathfrak{m}$ is exactly $\Ff_q$ (rather than a finite field
  extension), and that the quotient map
  \[ A \twoheadrightarrow A/\mathfrak{m} \stackrel{\sim}{\rightarrow} \Ff_q\]
  is dual to the inclusion of the identity element $\Spec \Ff_q
  \hookrightarrow G$.  Indeed, the inclusion of the identity must
  correspond to \emph{some} homomorphism $f : A \to \Ff_q$.  Since $f$
  is a homomorphism of $\Ff_q$-algebras, $f$ is a left inverse to the
  structure map $\Ff_q \to A$, and so $f$ is surjective.  The kernel
  is a prime ideal, necessarily $\mathfrak{m}$.

  Now by properties of Artinian local rings, the maximal ideal
  $\mathfrak{m}$ is also the nilradical, so every $x \in \mathfrak{m}$
  is nilpotent.  In fact, $x^q = 0$ for all $x \in \mathfrak{m}$.
  Otherwise, the descending chain of ideals
  \begin{equation*}
    A \supsetneq (x) \supsetneq (x^2) \supsetneq \cdots \supsetneq
    (x^q) \supsetneq (0)
  \end{equation*}
  would contradict length $\le q$.

  So the $q$th power homomorphism on $A$ annihilates $\mathfrak{m}$,
  and must therefore be
  \begin{equation*}
    A \twoheadrightarrow A/\mathfrak{m} \stackrel{\sim}{\rightarrow}
    \Ff_q \hookrightarrow A
  \end{equation*}
  Thus the $q$th power Frobenius on $G$ must be $G \to \Spec \Ff_q \to
  G$, which is the zero endomorphism.
\end{proof}

\begin{fact}\label{divisibility-fact}
  Let $G$ be a commutative finite group scheme over a field $K$.
  \begin{itemize}
  \item Let $G'$ be a finite subgroup scheme.  Then the length of $G'$
    divides the length of $G$.
  \item Let $G^0$ denote the connected component of $G$.  Then
    $\ell(G^0) = \ell(G) / |G(K^{alg})|$.
  \end{itemize}
\end{fact}
The first point follows from Theorems~10.5-10.7 in \cite{pink}.  The
second point follows by the proof of Proposition~15.3 in \cite{pink}.

\begin{lemma}\label{loial}
  Suppose $A$ is a $g$-dimensional abelian variety over $\Ff_q$.
  Suppose $q > p^{2gi}$.  Let $r$ be the $p$-rank of $A$.  Let
  $\phi_q$ denote the $q$th power Frobenius endomorphism of $A$.  Then
  $\deg(\phi_q - p^i)$ is divisible by $p^{i(2g-r)}$.
\end{lemma}
\begin{proof}
  Take $\ell \ne p$.  By Fact~\ref{t-ell-magic-2}, $\deg(p^i) =
  p^{2gi}$ because $T_\ell A$ is a free $\Zz_\ell$-module of rank
  $2g$.  Let $G$ denote the scheme-theoretic kernel of the the
  multiplication-by-$p^i$ endomorphism of $A$.  Then $G$ is a finite
  group scheme of length $\deg(p^i) = p^{2gi}$.  By definition of
  $p$-rank, $G(\Ff_q^{alg}) \approx (\Zz/p^i)^r$, so $G(\Ff_q^{alg})$
  has size $p^{ir}$.  Therefore, the connected component $G^0$ of $G$
  has length $p^{2gi}/p^{ir} = p^{i(2g-r)}$, by
  Fact~\ref{divisibility-fact}.

  The endomorphism $\phi_q : A \to A$ restricts to the $q$th-power
  Frobenius endomorphism on $G$ and $G^0$.  By assumption, $q >
  p^{2ig} \ge p^{i(2g-r)}$, and so $\phi_q$ annihilates $G^0$ by
  Lemma~\ref{local-artin}.

  Let $G'$ denote the kernel of $\phi_q - p^i$.  Then $G^0$ is a
  closed subgroup scheme of $G'$.  By Fact~\ref{divisibility-fact},
  \begin{equation*}
    \ell(G^0) = p^{i(2g-r)} \text{ divides } \ell(G) = \deg(\phi_q - p^i). \qedhere
  \end{equation*}
\end{proof}

\begin{proposition} \label{prop:h1}
  There is a computable function $h_2(p,s,g)$ with the following
  property.  Let $A$ be a $g$-dimensional abelian variety over
  $\Ff_q$, with $q = p^k > h_2(p,s,g)$.  Let $\phi_q$ denote the $q$th
  power Frobenius on $A$.  Let $r$ be the $p$-rank of $A$.  Then we
  can write the characteristic roots of $\phi_q$ as $\alpha_1, \ldots,
  \alpha_{2g}$, where
  \begin{itemize}
  \item $\alpha_1, \ldots, \alpha_r$ are the eigenvalues of
    $T_p(\phi_q) : T_p A \to T_p A$.
  \item $v_p(\alpha_i) > v_p(p^s)$ for $i \in \{r+1, r+2, \ldots, 2g\}$.
  \end{itemize}
\end{proposition}
\begin{proof}
  Define
  \begin{equation*}
    h_2(p,s,g) = \max\{ p^{2g \cdot h_1(2g,d',p,s)} : 0 \le d' \le 2g\},
  \end{equation*}
  where $h_1$ is as in Lemma~\ref{sparta-1}.
  Suppose the assumptions hold.  Then for any $1 \le i \le h_1(2g, 2g -
  r, p, s)$, we have
  \begin{equation*}
    q = p^k > h_2(p,s,g) \ge p^{2g \cdot h_1(2g,2g-r,p,s)} \ge
    p^{2gi}.
  \end{equation*}
  By Lemma~\ref{loial},
  \begin{equation*}
    v_p(\deg(\phi_q - p^i)) \ge v_p(p^{i(2g-r)}) \qquad \text{ for } i
    \le h_1(2g,2g-r,p,s).
  \end{equation*}
  Let $Q(x)$ be the rational polynomial whose roots
  are the $\alpha_i$.  By Corollary~\ref{poly-roots},
  \begin{equation*}
    \deg(\phi_q - p^i) = \prod_{i = 1}^{2g} (\alpha_i - p^i) = Q(p^i).
  \end{equation*}
  Thus
  \begin{equation*}
    v_p(Q(p^i)) \ge v_p(p^{i(2g-r)}) \qquad \text{ for } i \le
    h_1(2g,2g-r,p,s).
  \end{equation*}
  By definition of $h_1$ (Lemma~\ref{sparta-1}), it follows that at
  least $2g-r$ of the roots of $Q(x)$ have $p$-adic valuation at least
  $v_p(p^s)$.  Meanwhile, Lemma~\ref{the-other-r} gives $r$ roots
  $\beta_1, \ldots, \beta_r$, coming from the eigenvalues of
  $T_p(\phi_q)$.  Each of thse roots has valuation zero.  There can be
  no overlap between the $2g - r$ roots of valuation at least
  $v_p(p^s)$, and the $r$ roots coming from $T_p(\phi_q)$, so these
  together account for all $2g$ roots of $Q(x)$.
\end{proof}

\begin{corollary}\label{trace-count-2}
  There is a computable function $h(p,s,g)$ with the following
  property.  Let $C$ be a curve of genus $g$ over a finite field
  $\Ff_q$, and let $J$ be its Jacobian.  Suppose $q$ is a power of
  $p$, and $q > h(p,s,g)$.  Then
  \begin{equation*}
    |C(\Ff_q)| \equiv 1 - \Tr(\phi_q | J[p^s]) + \Tr(\phi_q |
    \Gg_m[p^s]) \pmod{p^s},
  \end{equation*}
  where the notation is as in Corollary~\ref{trace-count}.
\end{corollary}
\begin{proof}
  Take $h(p,s,g)$ to be the maximum of $h_2(p,s,g)$ and $p^s$.
  Suppose $q > h(p,s,g)$.  By Fact~\ref{curve-fact},
  \begin{equation*}
    |C(\Ff_q)| = 1 + q - \sum_{i = 1}^{2g} \alpha_i.
  \end{equation*}
  Working modulo $p^s$, the term $q$ vanishes, because $q > h(p,s,g)
  \ge p^s$.  Also, $q > h_2(p,s,g)$, so by Proposition~\ref{prop:h1},
  we may assume that
  \begin{itemize}
  \item $\alpha_1, \ldots, \alpha_r$ are the eigenvalues of $T_p(\phi_q)$
  \item $\alpha_{r+1},\ldots,\alpha_{2g}$ have valuation at least
    $v_p(p^s)$.
  \end{itemize}
  Working modulo $p^s$, we can therefore ignore
  $\alpha_{r+1},\ldots,\alpha_{2g}$.  Thus
  \begin{equation*}
    |C(\Ff_q)| \equiv 1 - \sum_{i = 1}^r \alpha_i \pmod{p^s}.
  \end{equation*}
  The right hand side is $1 - \Tr(\phi_q | J[p^s])$.  Finally, observe
  that $\Tr(\phi_q | \Gg_m[p^s])$ vanishes, because $T_p \Gg_m$ is
  free of rank 0.  (There is no $p$-torsion in the multiplicative
  group.)
\end{proof}

\section{Review of periodic difference fields}\label{pf-review}
In this section, we review the basic facts about periodic fields.  The
original source for these results is apparently Hrushovski's
hard-to-find \cite{the-lost-paper}.  We will follow an approach that
mimics the closely related case of ACFA (\cite{zoe-udi},
\cite{frobenius}).

Recall that a periodic field $(K_\infty,\sigma)$ is secretly a
multi-sorted structure $(K_1,K_2,\ldots)$ where $K_n$ is the fixed
field of $\sigma^n$ on $K_\infty$.  The multi-sorted structure has the
following functions and relations:
\begin{itemize}
\item The inclusion maps $K_n \hookrightarrow K_m$ when $n$ divides
  $m$
\item The difference field structure on each $K_n$
\end{itemize}

\subsection{Existentially closed periodic fields}

If $(K_\infty,\sigma)$ is a periodic field, then $K_n/K_1$ is a cyclic
Galois extension of degree at most $n$.  Say that $(K_\infty,\sigma)$
is \emph{non-degenerate} if $\Gal(K_n/K_1) \cong \Zz/n\Zz$ for each
$n$.  Equivalently, $K_n \not\subseteq K_m$ for any $m < n$.
\begin{lemma}\label{non-deg}
  If $(K_\infty,\sigma)$ is a non-degenerate periodic field and
  $(L_\infty,\sigma)$ extends $(K_\infty,\sigma)$, then the natural
  map
  \begin{equation*}
    \psi_n : L_1 \otimes_{K_1} K_n \to L_n
  \end{equation*}
  is an isomorphism of difference rings for all $n \in \Nn \cup
  \{\infty\}$.
\end{lemma}
\begin{proof}
  The $n = \infty$ case follows by taking the limit, so we may assume
  $n < \infty$.  The image of $\psi_n$ is the compositum $K_nL_1$.
  This is an intermediate field in the Galois extension $L_n/L_1$, so
  it must be $L_m$ for some $m$ dividing $n$.  By non-degeneracy, $K_n
  \not \subseteq L_m$ for any $m < n$.  Thus $K_nL_1 = L_n$ and the map is
  surjective.  Non-degeneracy of $K_\infty$ implies non-degeneracy of
  $L_\infty$, and so
  \begin{equation*}
    [K_n : K_1] = n = [L_n : L_1].
  \end{equation*}
  Counting dimensions, $\psi_n$ must be injective.
\end{proof}
Recall that a field extension $L/K$ is \emph{regular} if $L \otimes_K
K^{alg}$ is a domain, or equivalently, a field.  A field $K$ is
\emph{pseudo algebraically closed} (PAC) if $K$ is relatively
existentially closed in every regular extension.  An equivalent
condition is that $V(K) \ne \emptyset$ for every geometrically
integral variety $V$ over $K$.  This property is first-order
(\cite{field-arithmetic}, Proposition~10.9).
\begin{proposition}\label{acpf-char}
  A periodic field $(K_\infty,\sigma)$ is existentially closed if and
  only if
  \begin{enumerate}
  \item \label{ec1} $K_\infty \models \mathrm{ACF}$,
  \item \label{ec2} $(K_\infty,\sigma)$ is non-degenerate, and
  \item \label{ec3} $K_1$ is PAC.
  \end{enumerate}
\end{proposition}
\begin{proof}
  Suppose (\ref{ec1}) fails.  Extend $\sigma$ to an automorphism
  $\sigma'$ of $K_\infty^{alg}$.  Then $(K_\infty,\sigma)$ fails to be
  existentially closed in $(K_\infty^{alg},\sigma')$.

  Suppose (\ref{ec2}) fails, so that $K_n = K_m$ for some $m < n$.
  Let $\sigma'$ be the automorphism of $K'_\infty :=
  K_\infty(x_1,\ldots,x_n)$ extending $\sigma$ and mapping \[x_1
  \mapsto x_2 \mapsto \cdots \mapsto x_n \mapsto x_1.\] Then
  $(K_\infty,\sigma)$ is not existentially closed in
  $(K'_\infty,\sigma')$.  Indeed, the equation $\sigma^n(x) = x \ne
  \sigma^m(x)$ has a solution in $K'_n$ but not $K_n$.

  Suppose (\ref{ec3}) fails, so $K_1$ is not existentially closed in
  some regular extension $L/K_1$.  The difference ring $L_\infty := L
  \otimes_{K_1} K_\infty$ is a field by regularity of $L/K_1$.  Then
  $L_\infty$ is a periodic field extending $K_\infty$, and $K_\infty$
  is not existentially closed in $L_\infty$ because $K_1$ is not
  existentially closed in $L_1$.

  Finally, suppose (\ref{ec1}-\ref{ec3}) all hold.  Let $L_\infty$ be
  a periodic field extending $K_\infty$.  Let $K^*_\infty$ be a big
  ultrapower of $K_\infty$ (in the language of periodic fields, not difference fields).  It suffices to embed $L_\infty$ into
  $K^*_\infty$ over $K_\infty$.  Note that
  \begin{equation*}
    K^*_\infty = K^*_1 \otimes_{K_1} K_\infty = K^*_1 \otimes_{K_1} K_1^{alg}.
  \end{equation*}
  The first equality holds by Lemma~\ref{non-deg} and (\ref{ec2}); the
  second equality holds by (\ref{ec1}) and the general fact that
  $K_\infty/K_1$ is algebraic.  Similarly
  \begin{equation*}
    L_\infty = L_1 \otimes_{K_1} K_\infty = L_1 \otimes_{K_1} K_1^{alg}.
  \end{equation*}  
  Then $L_1/K_1$ is regular, so $K_1$ is existentially closed in $L_1$
  by (\ref{ec3}).  It follows that $L_1$ embeds into $K^*_1$ over
  $K_1$.  Tensoring with $K_\infty$, this gives the desired embedding of
  periodic fields:
  \begin{equation*}
    L_\infty = L_1 \otimes_{K_1} K_\infty \hookrightarrow K^*_1
    \otimes_{K_1} K_\infty = K^*_\infty. \qedhere
  \end{equation*}
\end{proof}
The conditions of Proposition~\ref{acpf-char} are first order, in
spite of appearances to the contrary.
\begin{definition}
  \emph{ACPF} is the theory of existentially closed periodic fields.
  In other words, ACPF is the model companion of periodic fields.
\end{definition}
The name ``ACPF'' is not standard, and is chosen by analogy with ACFA.

If $(K,\sigma)$ is a periodic field, let $\Abs(K)$ denote the
``absolute numbers,'' the relative algebraic closure of the prime
field in $K$.  We can regard $\Abs(K)$ as a substructure of $K$.  The
field $\Abs(K)$ is algebraically closed whenever $K$ is.
\begin{lemma}\label{ee-criterion}
  Two models $K_1, K_2 \models \ACPF$ are elementarily
  equivalent if and only if $\Abs(K_1) \cong \Abs(K_2)$.  More
  generally, if $F$ is a substructure of $K_1$ and $F = F^{alg}$, then
  any embedding of $F$ into $K_2$ is a partial elementary map from
  $K_1$ to $K_2$.
\end{lemma}
\begin{proof}
  The proof is the same as for ACFA (\cite{zoe-udi}, Theorem 1.3).
  Let $L = \Frac(K_1 \otimes_F K_2)$.  Then $L$ is a periodic
  field amalgamating $K_1$ and $K_2$ over $F$.  By
  companionability, $K_1$ and $K_2$ have the same type over $F$.
\end{proof}
Recall that a field is \emph{pseudofinite} if it perfect, PAC, and has
absolute Galois group $\hat{\Zz}$.  Models of ACPF are essentially
pseudofinite fields with a choice of a generator of the Galois group:
\begin{proposition}\label{acpf-is-pseuf}
  If $K$ is pseudofinite and $\sigma$ is a topological generator of
  $\Gal(K)$, then $(K^{alg},\sigma) \models \ACPF$.  The
  periodic field $(K^{alg},\sigma)$ and the pseudofinite field $K$ are
  bi-interpretable after naming parameters.  All models of ACPF arise
  in this way from pseudofinite fields.
\end{proposition}
\begin{proof}
  Except for bi-interpretability, this follows from
  Proposition~\ref{acpf-char}.  Note that ``$(K^{alg},\sigma)$'' is
  really the multisorted structure $(K_1,K_2,\ldots)$ where $K_n$ is
  the degree $n$ extension of $K$.  This can be interpreted in $K$ by
  choosing a basis for each $K_n$ and interpreting $K_n$ as $K^n$.
  Conversely, $K$ is $K_1$.
\end{proof}

\subsection{Definable sets}
The following standard fact is an easy application of compactness:
\begin{fact}\label{kpct-fact}
  Let $\Mm$ be a monster model.  Let $A \subseteq \Mm$ be small.  Let
  $\mathcal{P}$ be a collection of $A$-definable subsets of $\Mm^n$
  closed under positive boolean combinations.  Suppose the following
  holds:
  \begin{quote}
    For every $a, b \in \Mm^n$, if
    \[ \forall X \in \mathcal{P} : a \in X \implies b \in X,\]
    then $\tp(a/A) = \tp(b/A)$.
  \end{quote}
  Then every $A$-definable subset of $\Mm^n$ is in $\mathcal{P}$.
\end{fact}
We shall need the following geometric form of almost quantifier
elimination.  Recall that a morphism $f : V_1 \to V_2$ of
$K$-varieties is \emph{quasi-finite} if the fibers of the map
$V_1(K^{alg}) \to V_2(K^{alg})$ are finite.
\begin{proposition}\label{all-in-zero}
  Let $(\Mm,\sigma)$ be a model of ACPF.  Let $(K_\infty,\sigma)$ be a
  non-degenerate substructure, with $K_1$ perfect.  Let $X$ be a
  $K_\infty$-definable subset of $\Mm_1^n$.  Then $X$ is the image of
  $V(\Mm_1) \to \Aa^n(\Mm_1)$ for some quasi-finite morphism $V \to
  \Aa^n$ of $K_1$-varieties.
\end{proposition}
\begin{proof}
  Replacing $\Mm$ with an elementary extension, we may assume $\Mm$ is
  $|K_\infty|^+$-saturated.  Let $\mathcal{P}$ be the class of
  definable subsets of $\Mm_1^n$ of the specified form.  We need to
  show that $\mathcal{P}$ contains every $K_\infty$-definable subset
  of $\Mm_1^n$.

  Note that $\mathcal{P}$ is closed under finite unions, because we
  can form coproducts $V_1 \sqcup V_2$ in the category of
  $K_1$-varieties.  Similarly, $\mathcal{P}$ is closed under finite
  intersections, because of fiber products $V_1 \times_{\Aa^n} V_2$.
  Therefore, we can use Fact~\ref{kpct-fact}.  Let $a, b$ be two
  points in $\Mm_1^n$.  Suppose that for every $X \in \mathcal{P}$,
  \[ a \in X \implies b \in X.\]
  We must show $\tp(a/K_\infty) = \tp(b/K_\infty)$.  Let
  $(K_1(a)^{alg})_1$ denote the fixed field of the periodic difference
  field $K_1(a)^{alg} \subseteq \Mm_\infty$.
  \begin{claim}
    Let $c$ be an $m$-tuple from $(K_1(a)^{alg})_1$ and $\phi(x;y)$ be
    a quantifier-free $\mathcal{L}_{rings}(K_1)$-formula such that
    $\phi(a;c)$ holds.  Then there is an $m$-tuple $d$ from $\Mm_1$
    such that $\phi(b;d)$ holds.
  \end{claim}
  \begin{claimproof}
    Strengthening $\phi(x;y)$, we may assume that
    \begin{itemize}
    \item $\phi(x;y)$ witnesses that $y \in K_1(x)^{alg}$.
    \item $\phi(\Mm_\infty)$ defines a locally closed subvariety $W$ of
      $\Aa^{n+m}$.
    \end{itemize}
    Then the projection $W \to \Aa^n$ is a quasi-finite morphism of
    varieties over $K_1$.  Let $X \in \mathcal{P}$ be the image of
    $W(\Mm_1) \to \Aa^n(\Mm_1)$.  Then
    \begin{equation*}
      (a;c) \in W(\Mm_1) \implies a \in X \implies b \in X \implies
      (b;d) \in W(\Mm_1)
    \end{equation*}
    for some $m$-tuple $d \in \Mm_1$.
  \end{claimproof}
  By saturation, the Claim holds even when $c$ is an infinite tuple and $\phi(x;y)$ is a type.  Letting $c$ enumerate $(K_1(a)^{alg})_1$ and $\phi(x;y)$ be the complete type of $(a,c)$ over $K_1$, we obtain an embedding of fields
  \[ (K_1(a)^{alg})_1 \hookrightarrow \Mm_1\]
  mapping $a$ to $b$ and $K_1$ to $K_1$ pointwise.  By
  Lemma~\ref{non-deg}, we can apply the functor $- \otimes_{K_1}
  K_\infty$ and obtain an embedding of periodic fields
  \[ K_1(a)^{alg} \hookrightarrow \Mm_\infty\]
  sending $a$ to $b$, and $K_\infty$ to $K_\infty$ pointwise.  By
  Lemma~\ref{ee-criterion}, this is a partial elementary map, so
  $\tp(a/K_\infty) = \tp(b/K_\infty)$.
\end{proof}
In Proposition~\ref{all-in-zero}, note that $\dim(V) \le n$, because
the geometric fibers of $V \to \Aa^n$ are finite.
\begin{lemma}\label{qfree-to-curves-v2}
  Let $K$ be a pseudofinite field, and $V$ be a 1-dimensional variety
  over $K$.  In other words, $V(K^{alg})$ is 1-dimensional as a
  definable set in $K^{alg}$.  Then there exist
  curves\footnote{Geometrically irreducible, smooth, and projective as
    always.} $C_1, C_2, \ldots, C_n$ and a definable bijection between
  a cofinite subset of $V(K)$ and a cofinite subset of $\coprod_{i =
    1}^n C_i(K)$.
\end{lemma}
\begin{proof}
  Replacing $V$ with a closed subvariety, we may assume $V(K)$ is
  Zariski dense in $V(K^{alg})$.  This ensures that the irreducible
  components of $V$ are geometrically irreducible.  Let
  \[ V = D_1 \cup \cdots \cup D_n\]
  be the decomposition of $V$ into irreducible components.  Each $D_i$
  is a 1-dimensional irreducible variety over $K$, hence has a unique
  smooth projective model $C_i$.  Outside of finitely many exceptional
  points, there is a canonical bijection
  \[ D_1(K) \cup \cdots \cup D_n(K) \cong C_1(K) \sqcup \cdots \sqcup C_n(K). \qedhere\]
\end{proof}

\subsection{The theory of Frobenius periodic fields}\label{sec:tacpf}
Recall the Frobenius periodic fields $\Fr^q_\infty = (\Ff_q^{alg},\phi_q)$,
where $\phi_q$ is the $q$th power Frobenius.  There is an analogy
\begin{quote}
  finite fields : pseudofinite fields :: Frobenius periodic fields : e.c.\ periodic fields.
\end{quote}
Ax showed that a field $K$ is pseudofinite if and only if it is
elementarily equivalent to a non-principal ultraproduct of finite
fields.  The analogous thing happens here.
\begin{definition}
  $\TACPF$ is the theory of periodic fields $K_\infty$ such that
  \begin{enumerate}
  \item \label{tec1} $K_\infty \models \mathrm{ACF}$
  \item \label{tec2} $K_\infty$ is non-degenerate
  \item \label{tec3} $K_1$ is a model of the theory $T_{fin}$ of
    finite fields.
  \item \label{tec4} If $K_1$ has size $q < \infty$, then $\sigma$ acts as the $q$th power
    Frobenius on $K_\infty$.
  \end{enumerate}
\end{definition}
Ax showed that the models of $T_{fin}$ are exactly the finite and
pseudofinite fields.
\begin{lemma}\label{axish-0}
  The models of $\TACPF$ are exactly
  \begin{itemize}
  \item Models of $\ACPF$
  \item Frobenius periodic fields
  \end{itemize}
\end{lemma}
\begin{proof}
  If $(K_\infty,\sigma) \models \ACPF$, then axioms (\ref{tec1}) and
  (\ref{tec2}) hold by definition, (\ref{tec3}) holds because $K_1$ is
  pseudofinite by Proposition~\ref{acpf-is-pseuf}, and (\ref{tec4}) is
  vacuous, as pseudofinite fields are infinite.  If $(K,\sigma)$ is
  the $q$th Frobenius periodic field $\Fr^q$, then all the axioms are
  trivial.  Conversely, suppose $(K_\infty,\sigma) \models \TACPF$.
  If $|K_1| = q < \infty$, then axiom (\ref{tec1}) forces $K_\infty \cong \Ff_q^{alg}$ and axiom (\ref{tec4}) forces $(K_\infty,\sigma) \cong
  \Fr^q$.  If $K_1$ is infinite, then (\ref{tec3})
  forces $K_1$ to be pseudofinite, hence PAC.  Then (\ref{tec1}) and
  (\ref{tec2}) ensure $(K_\infty,\sigma) \models \ACPF$.
\end{proof}
\begin{corollary}\label{axish-1}
  If $(K_\infty,\sigma)$ is a non-principal ultraproduct of Frobenius
  periodic fields, then $(K_\infty,\sigma) \models \ACPF$.
\end{corollary}
\begin{lemma}\label{axish-2}
  If $(K_\infty,\sigma) \models \ACPF$ and $K_\infty$ has
  characteristic 0, then $(K_\infty,\sigma)$ is elementarily
  equivalent to an ultraproduct of Frobenius periodic fields $\Fr^p$
  with $p$ \emph{prime}.
\end{lemma}
\begin{proof}
  For each prime $p$, let $\tilde{F}_p$ be the periodic field
  $(\Qq_p^{un},\sigma)$, where $\Qq_p^{un}$ is the maximal unramified
  algebraic extension of $\Qq_p$, and $\sigma$ induces the $p$th power
  Frobenius on the residue field.  By the Chebotarev density theorem,
  there is a non-principal ultraproduct $(\tilde{F}^*,\sigma)$ of
  $\tilde{F}_p$ such that
  \[ (\Abs(\tilde{F}^*),\sigma) \cong(\Abs(K),\sigma).\]
  Now $\tilde{F}^*$ has a $\sigma$-invariant valuation whose residue
  field is an ultraproduct $F^*$ of Frobenius periodic fields $\Fr^p$.
  Then $F^*$ has characteristic 0, the valuation is equicharacteristic
  0, and the residue map gives an isomorphism
  \[ (\Abs(\tilde{F}^*),\sigma) \cong (\Abs(F^*),\sigma).\]
  By Lemma~\ref{ee-criterion} and Corollary~\ref{axish-1}, $(K,\sigma) \equiv
  (F^*,\sigma)$.
\end{proof}
\begin{lemma}\label{axish-3}
  If $(K_\infty,\sigma) \models \ACPF$ and $K$ has characteristic $p >
  0$, then $K$ is elementarily equivalent to a non-principal
  ultraproduct of Frobenius periodic fields $\Fr^q$, with $q$ ranging
  over powers of $p$.
\end{lemma}
\begin{proof}
  Similar to Lemma~\ref{axish-2}, but easier (no valuations or Chebotarev).
\end{proof}
\begin{proposition}\label{acpf-is-ultraprod}
\hfill
  \begin{enumerate}
  \item A periodic field $(K,\sigma)$ is existentially closed if and
    only if it is elementarily equivalent to a non-principal
    ultraproduct of Frobenius periodic fields.
  \item The elementary class generated by Frobenius periodic fields
    consists of the Frobenius periodic fields and existentially closed
    periodic fields.
  \item $\TACPF$ is the theory of Frobenius periodic fields.
  \end{enumerate}
\end{proposition}
Let $T_{prime}$ be the theory of finite prime fields $\Ff_p$.  Ax
showed that the models of $T_{prime}$ are exactly the following:
\begin{itemize}
\item The finite prime fields $\Ff_p$
\item The pseudofinite fields of characteristic 0.
\end{itemize}
Analogously, one can show:
\begin{proposition}
\hfill
  \begin{enumerate}
  \item A periodic field $(K,\sigma)$ is existentially closed of
    characteristic 0 if and only if it is elementarily equivalent to a
    non-principal ultraproduct of prime Frobenius periodic fields.
  \item The elementary class generated by prime Frobenius periodic
    fields consists of:
    \begin{itemize}
    \item Prime Frobenius periodic fields
    \item Existentially closed periodic fields of characteristic 0
    \end{itemize}
  \item The theory of prime Frobenius periodic fields is
    axiomatized by $\TACPF$ and the statement that $K_1 \models
    T_{prim}$.
  \end{enumerate}
\end{proposition}
We leave the proof as an exercise to the reader.

\section{Proof of the main theorem}\label{sec:main}
\subsection{The implicit definition}\label{sec-beth}
We will use the following forms of Beth implicit definability:
\begin{fact}[= Theorem~6.6.4 in \cite{hodges-long}] \label{beth-version}
  Let $L^+ \supseteq L^-$ be languages.  Let $T^-$ be an $L^-$ theory
  and $T^+$ be an $L^+$ theory extending $T^-$.  Let $\phi(x)$ be an
  $L^+$ formula.  Suppose that whenever $N \models T^-$, and $M_1^+$
  and $M_2^+$ are two expansions of $N$ to a model of $T^+$, that
  $\phi(M_1^+) = \phi(M_2^+)$.  Then there is an $L^-$-formula
  $\psi(x)$ such that $T^+ \vdash \phi \leftrightarrow \psi$.
\end{fact}
\begin{corollary}\label{beth-2}
  Let $L^+ \supseteq L^-$ be languages.  Let $T^-$ be an $L^-$ theory
  and $T^+$ be an $L^+$ theory extending $T^-$.  Suppose that
  \begin{itemize}
  \item $T^-$ is the theory of some (non-elementary) class $\mathcal{C}$ of
    $L^-$-structures.
  \item Every model of $T^-$ has at most one expansion to a model of
    $T^+$.
  \item Every model in $\mathcal{C}$ has at least one expansion to a
    model of $T^+$.
  \end{itemize}
  Then every model of $T^-$ has a unique expansion to a model of
  $T^+$, and $T^+$ is a definitional expansion of $T^-$.
\end{corollary}
\begin{proof}
  If $M \models T^-$, then $M$ is elementarily equivalent to an
  ultraproduct
  \begin{equation*}
    M \equiv M' = \prod_{i \in I} M_i /\mathcal{U}
  \end{equation*}
  of structures $M_i \in \mathcal{C}$.  Each $M_i$ can be expanded to
  a model of $T^+$, so the same holds for the ultraproduct $M'$.  By
  Fact~\ref{beth-version} and the assumptions, the $T^+$-structure on
  $M'$ is 0-definable from the $T^-$-structure.  Therefore the
  $T^+$-structure transfers along the elementary equivalence $M'
  \equiv M$, giving a $T^+$-structure on $M$.  So every model of $T^-$
  expands to a model of $T^+$ in a unique way.  By
  Fact~\ref{beth-version}, $T^+$ is a definitional expansion of $T^-$.
\end{proof}
We will apply both versions of implicit definability in the following
context:
\begin{itemize}
\item The language $L^-$ is the language of periodic fields.
\item The theory $T^-$ is $\TACPF$, the theory of Frobenius periodic
  fields as in \S\ref{sec:tacpf}.
\item $\mathcal{C}$ is the class of Frobenius periodic fields.
\item The language $L^+$ is the expansion of $L^-$ by a new predicate
  $P_{\phi,n,k}(\vec{y})$ for every formula
  $\phi(\vec{x};\vec{y}) \in L^-$, every $n \in \Nn$, and every $k \in
  \Zz/n\Zz$.
\end{itemize}
The theory $T^+$ is $T^-$ plus the following axioms:
\begin{enumerate}
\item \label{t-plus-1} For every $\phi$, $n$, and
  $b$, there is a unique $k \in \Zz/n\Zz$ such that $P_{\phi,n,k}(b)$ holds.
\item \label{t-plus-2} If $\phi(K;b) = \phi'(K;b')$, then
  \[ P_{\phi,n,k}(b) \iff P_{\phi',n,k}(b').\]
\item \label{t-plus-3} If $X$ is a definable set $\phi(K;b)$, let
  $\chi_n(X)$ denote the unique $k$ such that $P_{\phi,n,k}(b)$
  holds.  (This is well-defined by (\ref{t-plus-1}) and
  (\ref{t-plus-2}).)  Then $\chi_n$ is a strong Euler characteristic
  for each $n$.
\item \label{t-plus-4} The diagram
  \begin{equation*}
    \xymatrix{\Def(M) \ar[r]^{\chi_n} \ar[rd]^{\chi_m} & \Zz/n\Zz \ar[d]
      \\ & \Zz/m\Zz}
  \end{equation*}
  commutes when $m$ divides $n$.
\item \label{t-plus-5} Let $C$ be a genus-$g$ curve
  over $K_1$, and let $J$ be its Jacobian.  Let $p^k$ be a prime
  power.  Let $h$ be the function from Corollary~\ref{trace-count-2}.
  If $\characteristic(K) \ne p$ or if $|K_1| > h(g,p,k)$, then
  $\chi_{p^k}(C(K_1))$ is given by the formula
  \begin{equation*}
    \chi_{p^k}(C(K_1)) = 1 - \Tr(\sigma | J[p^k]) + \Tr(\sigma | \Gg_m[p^k]).
  \end{equation*}
  Here, if $G$ is a commutative group variety over $K_1$, then
  $\Tr(\sigma | G[n])$ denotes the trace of the action of $\sigma$ on
  the group of $n$-torsion in $G(K_\infty)$.
\end{enumerate}
Axioms \ref{t-plus-1}-\ref{t-plus-4} encode the statement that $\chi$
is a $\hat{\Zz}$-valued strong Euler characteristic, and Axiom
\ref{t-plus-5} determines its value on curves.  We discuss why Axiom
\ref{t-plus-5} is first-order in \S\ref{sec:interlude}.

\subsection{Uniqueness}
The ``existence'' part of Corollary~\ref{beth-2} has already been verified:
\begin{proposition}\label{finites-do-work}
  If $\Fr^q$ is a Frobenius periodic field, and $\chi$ is the counting
  Euler characteristic, then $\chi$ satisfies $T^+$.  In particular,
  $\Fr^q$ admits an expansion to a model of $T^+$.
\end{proposition}
\begin{proof}
  Examining the definition of $T^+$, axioms
  (\ref{t-plus-1})-(\ref{t-plus-4}) merely say that $\chi$ is a
  $\hat{\Zz}$-valued strong Euler characteristic, which is trivial.
  Axiom (\ref{t-plus-5}) holds by Corollaries~\ref{trace-count} and
  \ref{trace-count-2}.
\end{proof}
Therefore, it remains to prove the ``uniqueness'' part.
Our goal is to show that on any $(K,\sigma) \models \TACPF$,
there is at most one $\hat{\Zz}$-valued Euler characteristic
satisfying the axioms of $T^+$.  Until
Proposition~\ref{there-can-only-be-one}, we will restrict our
attention to models of $\ACPF$.
\begin{remark}\label{silly}
  In Axiom~\ref{t-plus-5} of $T^+$, the condition ``$|K_1| >
  h(g,p,k)$'' is automatic when $K_1$ is infinite, i.e., when
  $(K_\infty,\sigma) \models \ACPF$.  Therefore, for models of ACPF,
  Axiom~\ref{t-plus-5} says the following: for any curve $C$ over
  $K_1$ with Jacobian $J$,
  \[ \chi_{p^k}(C(K_1)) = 1 - \Tr(\sigma | J[p^k]) + \Tr(\sigma | \Gg_m[p^k]).\]
  By the Chinese remainder theorem, this formula determines
  $\chi_n(C)$ for any $n$.
\end{remark}
\begin{lemma}\label{almost-unique-lemma-1}
  Let $(K_\infty,\sigma)$ be a model of $\ACPF$, admitting two
  expansions to a model of $T^+$.  Let $\chi$ and $\chi'$ be the
  corresponding $\hat{\Zz}$-valued strong Euler characteristics.  Then
  $\chi(X) = \chi'(X)$ for every unary definable set $X \subseteq
  K_1$.
\end{lemma}
\begin{proof}
  Say that a definable set is \emph{good} if $\chi(X) = \chi'(X)$.  Finite sets
  are good.  If $X$ is in definable bijection with $Y$ and $X$ is
  good, then so is $Y$.  A disjoint union of two good sets is good.
  If $S$ is a cofinite subset of $X$, then $S$ is good if and only if
  $X$ is good.  Consequently, if a cofinite subset of $X$ is in
  definable bijection with a cofinite subset of $Y$, then $X$ is good
  if and only if $Y$ is good.

  If $C$ is a curve over $K_1$, then $C(K_1)$ is
  good, by Remark~\ref{silly}.  Any disjoint union of sets of this
  form is also good.  By Lemmas~\ref{qfree-to-curves-v2}, the set
  $V(K_1)$ is good for any 1-dimensional variety $X$ over $K_1$.

  Now let $X$ be a definable subset of $(K_1)^1$. By
  Proposition~\ref{all-in-zero}, $X$ is the image of $V_1(K_1) \to
  \Aa^1(K_1)$ for some morphism $V_1 \to \Aa^1$ of $K_1$-varieties
  with geometrically finite fibers.  Let $V_n$ denote the $n$-fold
  fiber product
  \begin{equation*}
    \underbrace{V_1 \times_{\Aa^1} V_1 \times_{\Aa^1} \cdots
      \times_{\Aa^1} V_1.}_{n \textrm{ times}}
  \end{equation*}
  Each of the morphisms $V_n \to \Aa^1$ has geometrically finite
  fibers, so each variety $V_n$ is 1-dimensional.  Hence each set
  \[ Y_n := V_n(K_1)\]
  is good.  Note that $Y_n$ is the $n$-fold fiber product of $Y_1$
  over $X$.

  Let $m$ be a bound on the size of the fibers of $Y_1 \to X$.  For $1
  \le k \le m$, let $X_k$ denote the set of $a \in X$ such that
  $f^{-1}(a)$ has size $m$.  Let $\alpha_k$ and $\beta_k$ denote
  $\chi(X_k)$ and $\chi'(X_k)$.

  Because $\chi$ and $\chi'$ are strong Euler characteristics,
  \begin{align*}
    \chi(Y_n) &= \sum_{k = 1}^m \alpha_k k^n \\
    \chi'(Y_n) &= \sum_{k = 1}^m \beta_k k^n 
  \end{align*}
  for all $n$.  As the $Y_n$'s are good,
  \begin{equation*}
    \sum_{k = 1}^m \alpha_k k^n = \sum_{k = 1}^m \beta_k k^n
  \end{equation*}
  for $n = 1, \ldots, m$.  By invertibility of the Vandermonde matrix
  $\langle k^n \rangle_{1 \le k \le m,~1 \le n \le m}$, and the fact
  that $\hat{\Zz}$ has no $\Zz$-torsion, it follows that $\alpha_k =
  \beta_k$ for all $k$.  Consequently,
  \begin{equation}
    \chi(X) = \sum_{k = 1}^m \alpha_k = \sum_{k = 1}^m \beta_k =
    \chi'(X).
  \end{equation}
  Therefore $X$ is good.
\end{proof}
\begin{lemma}\label{almost-unique-lemma-n}
  For any $n$, the following statements are true:
  \begin{description}
  \item[$(S_n)$] Let $(K_\infty,\sigma)$ be a model of ACPF, admitting
    two expansions to a model of $T^+$.  Let $\chi$ and $\chi'$ be the
    corresponding $\hat{\Zz}$-valued strong Euler characteristics.
    Then $\chi(X) = \chi'(X)$ for every definable subset $X \subseteq
    (K_1)^n$.
  \item[$(T_n)$] If $(K_\infty,\sigma)$ is a model of ACPF, admitting
    an expansion to a model of $T^+$, and $\chi$ is the corresponding
    $\hat{\Zz}$-valued strong Euler characteristic, then for every
    definable family $\{X_a\}_{a \in Y}$ of subsets of $(K_1)^n$, for
    every $m \in \Nn$ and for every $k \in \Zz/m\Zz$, the set
    \[ \{ a \in Y(K) : \chi(X_a) \equiv k \pmod{m}\}\]
    is definable in the $L^-$-reduct $(K_\infty,\sigma)$.
  \end{description}
\end{lemma}
\begin{proof}
  Statement $S_1$ is Lemma~\ref{almost-unique-lemma-1}.  The
  implication $S_n \implies T_n$ follow by Beth implicit definability.
  It suffices to show
  \[ (S_1 \text{ and } S_n \text{ and } T_n) \implies S_{n+1}.\]
  Assume the left hand side.  Let $(K_\infty,\sigma), \chi, \chi',$
  and $X \subseteq K_1 \times (K_1)^n$ be as in the statement of
  $S_{n+1}$.  Fix $m \in \Nn$; we claim $\chi_m(X) = \chi'_m(X)$.  For $t \in K_1$, let
  \[ X_t = \{\vec{x} \in (K_1)^n : (t,\vec{x}) \in X\}\]
  By statements $S_n$ and $T_n$, the sets
  \begin{align*}
    Y_k &= \{t \in K_1 : \chi(X_t) \equiv k \pmod{m}\} \\
    Y'_k &= \{t \in K_1 : \chi'(X_t) \equiv k \pmod{m}\}
  \end{align*}
  are equal and definable.  Because $\chi$ and $\chi'$ are strong Euler characteristics,
  \begin{align*}
    \chi_m(X) &= \sum_{k \in \Zz/m\Zz} k \cdot \chi_m(Y_k) \\
    \chi'_m(X) &= \sum_{k \in \Zz/m\Zz} k \cdot \chi'_m(Y'_k).
  \end{align*}
  Then $\chi_m(Y_k) = \chi'_m(Y_k)$ by statement $S_1$, so putting
  things together, $\chi_m(X) = \chi'_m(X)$.  As $m$ was arbitrary,
  $S_n$ holds.
\end{proof}
\begin{proposition}\label{there-can-only-be-one}
  If $(K,\sigma)$ is a model of $\TACPF$, then there is at most one
  expansion of $(K,\sigma)$ to a model of $T^+$.
\end{proposition}
\begin{proof}
  If $(K,\sigma)$ is a Frobenius periodic field, then $K_\infty$ is
  essentially finite and there is at most one $\hat{\Zz}$-valued Euler
  characteristic.  So assume $(K_\infty,\sigma) \models \ACPF$.  Let
  $\chi, \chi'$ be two $\hat{\Zz}$-valued Euler characteristics
  satisfying $T^+$.  Note that the sort $K_n$ is in definable
  bijection with $(K_1)^n$.  If $X$ is any definable set in
  $K_\infty$, then $X$ is therefore in definable bijection with a
  definable subset $Y \subseteq (K_1)^m$ for some $m$.  By statement
  $S_m$ of Lemma~\ref{almost-unique-lemma-n},
  \[ \chi(X) = \chi(Y) = \chi'(Y) = \chi'(X). \qedhere\]
\end{proof}
By Corollary~\ref{beth-2} and Proposition~\ref{finites-do-work}, we conclude
\begin{proposition}\label{exactly-one}
  If $(K,\sigma)$ is a model of $\TACPF$, then there is a unique
  expansion of $(K,\sigma)$ to a model of $T^+$.
\end{proposition}
\begin{theorem-star}[Theorem~\ref{ecpdf-thm}]
  Let $\mathcal{C}$ be the class of Frobenius periodic fields and
  existentially closed periodic fields.  There is a $\hat{\Zz}$-valued
  strong Euler characteristic $\chi$ on $(K,\sigma)$ in $\mathcal{C}$ with the
  following properties:
  \begin{itemize}
  \item $\chi$ is uniformly 0-definable across $\mathcal{C}$.
  \item If $(K,\sigma)$ is a Frobenius periodic field, then $\chi$ is
    the counting Euler characteristic:
    \[ \chi(X) = |X|.\]
  \item If $(K,\sigma)$ is an ultraproduct of Frobenius periodic
    fields, then $\chi$ is the nonstandard counting Euler
    characteristic.
  \end{itemize}
\end{theorem-star}
\begin{definition}
  The \emph{canonical Euler characteristic} on $(K,\sigma) \models
  \TACPF$ is the $\hat{\Zz}$-valued Euler characteristic of
  Theorem~\ref{ecpdf-thm}.
\end{definition}
\begin{remark}
  The canonical Euler characteristic $\chi$ is the \emph{only}
  $\hat{\Zz}$-valued Euler characteristic that is uniformly
  0-definable across all models of $\TACPF$.  Indeed, if $\chi'$ is
  another uniformly definable Euler characteristic, and
  \[ \chi(\phi(K;b)) \ne \chi'(\phi(K;b))\]
  for some model $K$ and tuple $b$, then $K$ is elementarily
  equivalent to an ultraproduct of Frobenius periodic fields, so we
  can in fact take $K$ to be a Frobenius periodic field.  But
  Frobenius periodic fields are essentially finite, so $\chi$ and
  $\chi'$ must agree on $K$, a contradiction.
\end{remark}
\begin{remark}
  There are other uniformly 0-definable $\hat{\Zz}$-valued strong
  Euler characteristics on $\ACPF$.  For example,
  \[ (K_\infty,\sigma) \models \ACPF \implies (K_\infty,\sigma^{-1}) \models \ACPF,\]
  and the canonical Euler characteristic on $(K_\infty,\sigma^{-1})$
  induces a non-canonical Euler characteristic on $(K_\infty,\sigma)$.
  We shall have more to say about this in \S\ref{sec:anti}.
\end{remark}

\section{An interlude on definability and computability}\label{sec:interlude}
This section discusses some of the technical issues related to Axiom
(\ref{t-plus-5}) in the definition of $T^+$.  If one is willing to
sweep these issues under the rug, this section can be skipped.
\begin{lemma}\label{lem:tobecontinued}
  In the definition of $T^+$, Axiom (\ref{t-plus-5}) is expressible by
  first-order sentences.
\end{lemma}
\begin{proof}[Proof sketch]
  The assertion 
  \begin{quote}
    $J$ is the Jacobian of $C$
  \end{quote}
  can be expressed as
  \begin{quote}
    $J$ is a smooth projective group variety that is birationally
    equivalent (over $K_1$) to $\Sym^g C$, the $g$th symmetric product
    of $C$.
  \end{quote}
  Indeed, the Jacobian is a smooth projective group variety because it
  is an abelian variety, and it is birationally equivalent to $\Sym^g
  C$ by the construction of the Jacobian in \S V.1 of
  \cite{serre-jacobian}.  By Theorem~I.3.8 in \cite{milneAV}, any
  birational map between two projective group varieties extends to an
  isomorphism.

  Even the following statement is rather non-trivial to express:
  \begin{quote}
    $C$ is a (smooth projective) curve of genus $g$
  \end{quote}
  Smoothness can be witnessed by covering projective space with
  Zariski open patches on which $C$ is cut out by a system of
  equations whose matrix of partial derivatives has rank no higher
  than the codimension of $C$.  Geometric irreducibility can be
  witnessed as in the appendix of \cite{fls}.  Genus can be determined
  by counting zeros and poles on a meromorphic section of the tangent
  bundle.

  Hopefully, everything will be spelled out in greater detail in
  \cite{complementary-paper}.
\end{proof}
By being more careful, one can presumably show that the theory $T^+$
is not only first-order, but recursively axiomatized.  We have gone
too far out on a limb, so we state this as a conjecture:
\begin{conjecture}\label{horror}
  In the language of $T^+$, there are sentences $\tau_{g,p^k,n}$
  \textbf{depending recursively on the parameters}, such that the following are
  equivalent for any $g \ge 1$, any prime power $p^k$, and any
  structure $(K_\infty,\sigma,\chi)$ satisfying $\TACPF$ and Axioms
  \ref{t-plus-1}-\ref{t-plus-4} of $T^+$:
  \begin{enumerate}
  \item $(K_\infty,\sigma,\chi)$ satisfies $\bigwedge_{n = 1}^\infty
    \tau_{g,p^k,n}$.
  \item For every genus $g$ curve $C$ over $K_1$
    with Jacobian $J$,
    \begin{equation*}
      \chi_{p^k}(C(K_1)) = 1 - \Tr(\sigma | J[p^k]) + \Tr(\sigma | \Gg_m[p^k]).
    \end{equation*}
  \end{enumerate}
\end{conjecture}
\begin{lemma}\label{yanked}
  (Assuming Conjecture~\ref{horror}) The theory $\TACPF$ of
  \S\ref{sec:tacpf} and the theory $T^+$ of \S\ref{sec-beth} are
  recursively axiomatized.
\end{lemma}
\begin{proof}
  For $\TACPF$, this is mostly clear.  Axiom \ref{tec3}, saying that
  $K_1$ is a model of the theory of finite fields, is recursively
  axiomatized by Ax's theorem on the decidability of the theory of
  finite fields.

  The additional axioms of $T^+$ are plainly recursively enumerable,
  except for Axiom~\ref{t-plus-5}, which can be expressed by
  Conjecture~\ref{horror}.
\end{proof}

Conjecture~\ref{horror} is almost certainly true, by the method of
Lemma~\ref{lem:tobecontinued}.  To ``prove'' Conjecture~\ref{horror}, we
seemingly have three options:
\begin{enumerate}
\item \label{formal} A precise proof in terms of indexed families of
  formulas.
\item \label{informal} An informal ``proof'' in the style of
  Lemma~\ref{lem:tobecontinued}.
\item \label{cheating} A clever proof using subtle facts from
  algebraic geometry.
\end{enumerate}
There is something deeply unsatisfying about each of these approaches.
Approach~\ref{formal} is extremely tedious; writing out the details
would probably double the length of this paper.  Moreover, the details
would be an incomprehensible stew of indexed families of
multi-variable formulas.  For example, the statement that
underlies Conjecture~\ref{horror} is (almost) the following:
\begin{conjecture}\label{double-horror}
  There are formulas $\phi_{n,g}(\vec{x}),
  \psi_{n,g}(\vec{x},\vec{y}), \rho_{n,g}(\vec{x},\vec{z})$ depending
  recursively on $n$ and $g$ such that for any perfect field $K$ and
  any $g$, if $X, Y$ are two definable sets, then the following are
  equivalent:
  \begin{enumerate}
  \item There is a genus-$g$ curve $C/K$ with
    Jacobian $J$, such that $X$ is in definable bijection with $C(K)$
    and $Y$ is in definable bijection with $J(K)$
  \item There is some $n \in \Nn$ and some $\vec{a} \in \phi_{n,g}(K)$
    such that $X$ is in definable bijection with
    $\psi_{n,g}(\vec{a},K)$ and $Y$ is in definable bijection with
    $\rho_{n,g}(\vec{a},K)$.
  \end{enumerate}
\end{conjecture}
A proof written in this style would be nearly unreadable.

In contrast, Approach~\ref{informal} is excessively informal.  It is
hard to convince oneself that Lemma~\ref{lem:tobecontinued} is really a
proof of Conjecture~\ref{horror}, especially when one compares the
relative lengths of the informal proof and the precise proof.

What seems to be missing is a language that would assign precise
meaning to statements like the following:
\begin{itemize}
\item $\Pp^n$ depends nicely on $n$.
\item The family of Zariski closed sets in $\Pp^n$ depends nicely on
  $n$.
\item The family of smooth irreducible varieties of dimension $d$ depends
  nicely on $d$.
\item If $C$ is a curve, the family of meromorphic
  functions $C \to \Pp^1$ depends nicely on $C$.
\item If $C$ is a curve, if $f : C \to \Pp^1$ is
  meromorphic, and if $x \in C$, then the zero/pole-order of $f$ at
  $x$ depends nicely on $C, f, x$.
\item The family of curves of genus $g$ depends
  nicely on $g$.
\item If $C$ is a curve, then the Jacobian of $C$
  depends nicely on $C$.
\end{itemize}
Here, ``nicely'' is supposed to mean something like ``recursively
ind-definable, uniformly across all models.''

In future work (\cite{complementary-paper}), I hope to develop a
toolbox that makes this notion precise, enabling a smoother proof
of Conjecture~\ref{horror}.  My hope is that this toolbox will be
useful in other situations where one needs to verify the recursive
axiomatizability of conditions from algebraic geometry.

Finally, we consider Approach~\ref{cheating}---using clever tricks
from algebraic geometry to simplify the problem.  This approach
probably works, but is conceptually unsatisfying.  It \emph{should} be
possible to translate the informal proof into a precise proof that is
not too long.  It \emph{shouldn't} be necessary to resort to
non-elementary facts from algebraic geometry to overcome a syntactic
problem.
\begin{remark}
  An analogous situation appears when one verifies that the theory of
  PAC fields is recursively axiomatized.  The standard approach is to
  use Bertini's theorem to reduce to the case of curves, and then
  project into the plane to reduce to the case of plane curves (see \S
  10.2 in \cite{field-arithmetic}).  This is an instance of
  Approach~\ref{cheating}.
\end{remark}

\section{Further results}
\subsection{Uniform definability of the counting Euler characteristic}
Theorem~\ref{ecpdf-thm} implies that the counting Euler characteristic
is uniformly definable across all Frobenius periodic fields.  This can
be restated more explicitly as follows:
\begin{corollary}\label{uniform-elim}
    For any formula $\phi(x;y)$ in the language of periodic fields,
    any $n \in \Nn$, and any $k \in \Zz/n \Zz$, there is a formula
    $\psi_{\phi,n,k}(y)$ such that for any Frobenius periodic field
    $\Fr^q$ and any tuple $b$ from $\Fr^q$,
    \[ \Fr^q \models \psi_{\phi,n,k}(b) \iff |\phi(\Fr^q;b)| \equiv k \pmod{n}\]
\end{corollary}

\subsection{Evaluation on curves}
\begin{proposition}
  Let $(K_\infty,\sigma)$ be a model of $\ACPF$.  Let $C$ be a curve
  over $K_1$, and $J$ be the jacobian.  For any prime $\ell$ (possibly
  the characteristic), the $\ell$-adic component of $\chi(C(K_1))$ is
  determined by the trace of the action of $\sigma$ on the $\ell$-adic
  Tate modules of $J$ and the multiplicative group $\Gg_m$:
  \[ 1 - \Tr(\sigma | T_\ell J) + \Tr(\sigma | T_\ell \Gg_m).\]
\end{proposition}
\begin{proof}
  This follows directly from Axiom~\ref{t-plus-5} of $T^+$, and
  Remark~\ref{silly}.
\end{proof}
For $\ell \ne \characteristic(K)$, there should be a generalization
using $\ell$-adic etale cohomology:
\begin{conjecture}\label{weilish}
  Let $(K_\infty,\sigma)$ be a model of ACPF, let $V$ be a smooth projective variety
  over $K_1$, and let $\ell$ be a prime different from the
  characteristic.  Then the $\ell$-adic component of $\chi(V(K_1))$ is
  given by the formula
  \begin{equation*}
    \sum_{i = 0}^{2 \dim(V)} (-1)^i \Tr(\sigma^{-1} | H^i(V;\Qq_\ell)),
  \end{equation*}
  where $H^i(V;\Qq_\ell)$ denotes the $\ell$-adic cohomology:
  \begin{equation*}
    \Qq_\ell \otimes_{\Zz_\ell} \varprojlim_k H^i_{et}(V \times_{
      K_1} K_\infty; \Zz/\ell^k).
  \end{equation*}
\end{conjecture}

I suspect that Conjecture~\ref{weilish} is trivial with the right
tools.  If I understand correctly, the conjecture holds for Frobenius
periodic fields, because of Grothendieck's cohomological approach to
the Weil conjectures (described in Hartshorne \cite{hartshorne},
Appendix C, \S 3-4).  As long as Conjecture~\ref{weilish} can be
stated as a conjunction of first-order sentences, it transfers to
models of ACPF by Proposition~\ref{acpf-is-ultraprod}.  Thus, the only
thing needing verification is that the groups $H^i_{et}(V;
\Zz/\ell^k)$ depend definably on $V$.

\begin{remark}
  There should be a more general form of Conjecture~\ref{weilish} for
  arbitrary varieties, using cohomology with compact supports or
  intersection homology.
\end{remark}

\subsection{Pseudofinite fields}
\begin{lemma}\label{psf-case}
  Let $K$ be a pseudofinite field and $\sigma$ be a topological
  generator of $\Gal(K)$.  The canonical $\hat{\Zz}$-valued definable
  strong Euler characteristic on $(K^{alg},\sigma)$ restricts to an
  $\acl^{eq}(0)$-definable strong Euler characteristic on $K$.
\end{lemma}
\begin{proof}
  The structure $(K^{alg},\sigma)$ and the field $K$ have equivalent
  categories of (parametrically) definable sets, by the
  bi-interpretability of Proposition~\ref{acpf-is-pseuf}.  Therefore,
  the definable strong Euler characteristic on $(K^{alg},\sigma)$
  determines a definable strong Euler characteristic $\chi'$ on $K$.

  To prove $\acl^{eq}(0)$-definability of $\chi'$, we may pass to an
  elementary extension and assume $K$ and $(K^{alg},\sigma)$ are
  monster models.  The Euler characteristic $\chi'$ is not determined
  in an $\Aut(K)$-invariant way, because of the choice of $\sigma$.
  However, there are only boundedly many choices for $\sigma$.
  Therefore $\chi'$ has only boundedly many conjugates under
  $\Aut(K)$, so $\chi'$ is $\acl^{eq}(0)$-definable.
\end{proof}
\begin{theorem-star}[Theorem~\ref{psy-thm}]
  \hfill
  \begin{enumerate}
  \item Let $K = \prod_i K_i/\mathcal{U}$ be an ultraproduct of
    finite fields.  Then the nonstandard counting functions $\chi_n$
    are $\acl^{eq}(\emptyset)$-definable.\label{ultra-psy-copy}
  \item Every pseudofinite field admits an
    $\acl^{eq}(\emptyset)$-definable $\hat{\Zz}$-valued strong Euler
    characteristic.\label{rando-psy-copy}
  \end{enumerate}
\end{theorem-star}
\begin{proof}
  Part \ref{rando-psy-copy} is Lemma~\ref{psf-case}.  For part \ref{ultra-psy-copy}, given
  an ultraproduct $K = \prod_i \Ff_{q_i} / \mathcal{U}$, let
  $(L,\sigma) = \prod_i \Fr^{q_i} / \mathcal{U}$ be the
  corresponding ultraproduct of Frobenius periodic fields.  Then $K
  \cong L_1$.  The nonstandard counting functions on $K$ are induced
  by the canonical Euler characteristic on $(L_\infty,\sigma)$.
  Therefore the nonstandard counting functions on $K$ are
  $\acl^{eq}(0)$-definable, by Lemma~\ref{psf-case}.
\end{proof}

\subsection{Elimination of parity quantifiers}
For any $n \in \Nn$ and $k \in \Zz/n\Zz$, let $\mu^n_k x$ be a new
quantifier.  Interpret $\mu^n_k x : \phi(x)$ in finite structures as
\begin{quote}
  The number of $x$ such that $\phi(x)$ holds is congruent to $k$ mod $n$.
\end{quote}
In other words,
\begin{equation*}
  \left(M \models \mu^n_k \vec{x} : \phi(\vec{x},\vec{b})\right) \iff \left(|\{\vec{a} :
  M \models \phi(\vec{a},\vec{b})\}| \equiv k \pmod{n}\right).
\end{equation*}
For example,
\begin{itemize}
\item $\mu^2_0 x$ means ``there are an even number of $x$ such that\ldots''
\item $\mu^2_1 x$ means ``there are an odd number of $x$ such that\ldots''
\end{itemize}
We call $\mu^n_k$ \emph{generalized parity quantifiers}.

Let $\mathcal{L}^{\mu}_{rings}$ and $\mathcal{L}^{\mu}_{pf}$ be the
language of rings and the language of periodic fields, respectively,
expanded with generalized parity quantifiers.
\begin{proposition}[= Theorem~\ref{intro-deci-2}.\ref{id2a}]\label{parelim}
  Frobenius periodic fields uniformly eliminate generalized parity
  quantifiers.  If $\phi(\vec{x})$ is a formula in
  $\mathcal{L}^{\mu}_{pf}$, then there is a formula $\phi'(\vec{x})
  \in \mathcal{L}_{pf}$ such that for any Frobenius periodic field
  $\Fr^q$ and any tuple $\vec{a}$,
  \[ \Fr^q \models \phi(\vec{a}) \iff \Fr^q \models \phi'(\vec{a}).\]
\end{proposition}
\begin{proof}
  Proceed by induction on the complexity of $\phi(\vec{x})$.  We may
  assume $\phi(\vec{x})$ has the form
  \[ \mu^n_k \vec{y} : \psi(\vec{x},\vec{y}),\]
  for some formula $\psi(\vec{x},\vec{y}) \in \mathcal{L}_{pf}$.  In
  this case, we can eliminate $\mu^n_k$ by
  Corollary~\ref{uniform-elim}.
\end{proof}
\begin{example}\label{the-example}
  The $\mathcal{L}_{pf}^\mu$-sentence
  \begin{equation*}
    \tau \stackrel{def}{\iff} \mu^5_2 x \in K_1 : x = x
  \end{equation*}
  is equivalent in Frobenius periodic fields $\Fr^q$ to the
  $\mathcal{L}_{pf}$-sentence
  \begin{equation*}
    \tau' \stackrel{def}{\iff} 5 \ne 0 \wedge \forall x \in K_4 : (x^5
    = 1 \rightarrow \sigma(x) = x^2).
  \end{equation*}
  To see this, break into cases according to the congruence class of
  $q$ modulo 5.  Note that $\Fr^q \models \tau \iff q \equiv 2
  \pmod{5}$.
  \begin{itemize}
  \item If $q \equiv 0 \pmod{5}$, then $\Fr^q$ has characteristic 5, so
    $\tau'$ and $\tau$ are both false.
  \item If $q \equiv 2 \pmod{5}$, then $\Fr^q$ does not have characteristic 5, and
    \[ \forall x \in K_\infty : (x^5 = 1 \rightarrow x^q = x^2),\]
    so $\tau$ and $\tau'$ are both true.
  \item If $q \equiv j \pmod{5}$ for $j \ne 0, 2$, then $\Fr^q$ does not
    have characteristic 5.  Let $x$ be a primitive fifth root of
    unity.  Then $x \in K_4$, because $\Gal(K_1(x)/K_1)$ is a subgroup
    of $(\Zz/5\Zz)^\times$.  Also,
    \[ x^q = x^j \ne x^2,\]
    and so $\tau'$ is false.
  \end{itemize}
\end{example}
In contrast to Proposition~\ref{parelim}, generalized parity
quantifiers are \emph{not} eliminated in finite fields:
\begin{lemma}\label{lame}
  There is no $\mathcal{L}_{rings}$-sentence $\rho$ equivalent to
  the following $\mathcal{L}_{rings}^\mu$-sentence in every finite
  field:
  \[ \mu^5_2 x : x = x.\]
\end{lemma}
\begin{proof}
  Suppose $\rho$ exists.  Then the following are equivalent for any
  model $(K_\infty,\sigma) \models \TACPF$:
  \begin{itemize}
  \item $K_1$ satisfies $\rho$
  \item $K_\infty$ does not have characteristic 5, and the action of
    $\sigma$ on the fifth roots of unity is given by
    \begin{equation*}
      \sigma(\omega) = \omega^2.
    \end{equation*}
  \end{itemize}
  Now take $(K_\infty,\sigma)$ satisfying $\ACPF$ and the two
  equivalent conditions.  (For example, we can take $K_\infty$ to be
  a non-principal ultraproduct of $\Fr^p$ where $p$ ranges over primes
  congruent to 2 mod 5.  A non-principal ultrafilter exists by
  Dirichlet's theorem.)  Then $K_1$ satisfies $\rho$, and $\sigma$
  acts on the fifth roots of unity by squaring.  Consider a dual model
  \[ (K^\dag_\infty,\sigma) \cong (K_\infty,\sigma^{-1}).\]
  From the axioms of $\ACPF$, it is clear that
  $(K^\dag_\infty,\sigma) \models \ACPF$.  Since $\sigma$ acts on
  fifth roots by squaring, $\sigma^{-1}$ acts by cubing:
  \[ \sigma^{-1}(\omega) = \omega^3,\]
  as 2 and 3 are multiplicative inverses modulo 5.  So
  $(K^\dag_\infty,\sigma)$ does not satisfy the two equivalent
  conditions, and in particular, $K^\dag_1 \not \models \rho$.  But
  this is absurd, since $K^\dag_1$ is isomorphic as a field to $K_1$.
\end{proof}
\begin{remark}
  The proof of Lemma~\ref{lame} actually proves
  something stronger: parity quantifiers are not eliminated on the
  class of prime fields $\Ff_p$.
  The non-elimination of parity quantifiers in finite fields was
  originally proven in \cite{krajicek}, Theorem~7.3, using a slightly
  different method.  
\end{remark}

\subsection{Decidability}
Recall the theory $T^+$ of \S\ref{sec-beth}.  For the rest of this
section, we assume Conjecture~\ref{horror}.
\begin{lemma}\label{comply-convert}
  (Assuming Conjecture~\ref{horror}) There is a computable function which
  takes a formula $\phi(\vec{x})$ in the language of $T^+$ and outputs
  a formula $\phi'(\vec{x})$ in the language of periodic fields, such
  that $T^+ \vdash \phi \leftrightarrow \phi'$.
\end{lemma}
\begin{proof}
  By Lemma~\ref{yanked}, the theory $\TACPF$ of \S\ref{sec:tacpf} and
  the theory $T^+$ of \S\ref{sec-beth} are recursively axiomatized.
  
  For each $\phi$, an equivalent formula $\phi'$ exists by Beth
  implicit definability (Fact~\ref{beth-version}) and the existence
  and uniqueness of the expansion to $T^+$
  (Proposition~\ref{exactly-one}).  An algorithm can find $\phi'$ by
  searching all consequences of $T^+$ until it finds one of the form
  \[ \forall \vec{x} : \phi(\vec{x}) \leftrightarrow \phi'(\vec{x})\]
  with $\phi'$ a formula in the pure language of periodic fields.
\end{proof}
\begin{corollary}\label{conversion}
  (Assuming Conjecture~\ref{horror}.)
  \begin{enumerate}
  \item In Corollary~\ref{uniform-elim}, the formula $\psi_{\phi,n,k}$
    can be chosen to depend computably on $\phi$.
  \item In Proposition~\ref{parelim}, the elimination of generalized
    parity quantifiers can be carried out computably---the formula
    $\phi'$ can be chosen to depend computably on $\phi$.
  \end{enumerate}
\end{corollary}
\begin{proof}
  \hfill
  \begin{enumerate}
  \item Corollary~\ref{uniform-elim} is an instance of
    Lemma~\ref{comply-convert}, so the conversion can be done
    computably.
  \item As in the proof of Proposition~\ref{parelim}, one converts a
    $\mathcal{L}^{\mu}_{pf}$-formula into a pure
    $\mathcal{L}_{pf}$-formula by recursion on the formula. \qedhere
  \end{enumerate}
\end{proof}

\begin{theorem-star}[Theorems~\ref{intro-deci-2}.\ref{id2b} and \ref{intro-deci}]
  (Assuming Conjecture~\ref{horror}.)
  \begin{enumerate}
  \item The $\mathcal{L}^{\mu}_{pf}$-theory of Frobenius periodic
    fields is decidable.
  \item The $\mathcal{L}^{\mu}_{rings}$-theory of
    finite fields is decidable.
  \end{enumerate}
\end{theorem-star}
\begin{proof}
  First note that the ($\mathcal{L}_{pf}$-)theory of Frobenius
  periodic fields is decidable.  By
  Proposition~\ref{acpf-is-ultraprod}, the theory is completely
  axiomatized by $\TACPF$.  Therefore, the theory is computably
  enumerable.  The theory is also co-computably enumerable.  Indeed, a
  sentence $\tau$ is not part of the theory if and only if $\Fr^q
  \models \neg \tau$ for some $q$.  There is an algorithm taking $q$
  and $\tau$ and outputting whether or not $\Fr^q \models \tau$, because
  $\Fr^q$ is essentially finite.  So we can enumerate all the statements
  that fail in some Frobenius periodic field, which is the complement
  of the theory of Frobenius periodic fields.  Thus the theory of
  Frobenius periodic fields is decidable, as claimed.

  Now given a $\mathcal{L}^{\mu}_{pf}$-sentence $\tau$, we can
  computably convert it into an equivalent $\mathcal{L}_{pf}$-sentence
  $\tau'$, and use the previous paragraph to computably determine
  whether or not $\tau'$ holds in every Frobenius periodic field.
  This proves the first point.

  The second point follows, because there is a computable way to
  convert an $\mathcal{L}^{\mu}_{rings}$-sentence $\tau$ into a
  $\mathcal{L}^{\mu}_{pf}$-sentence $\tau'$ such that
  \begin{equation*}
    (K_\infty,\sigma) \models \tau' \iff K_1 \models \tau
  \end{equation*}
  for any essentially finite periodic field $(K_\infty,\sigma)$.
  Taking $K_\infty$ to be $\Fr^q$, we see that
  \begin{equation*}
    \Fr^q \models \tau' \iff \Ff_q \models \tau.
  \end{equation*}
  Therefore, $\tau$ holds in every finite field if and only if $\tau'$
  holds in every Frobenius periodic field.  Then we can apply the
  oracle for the first point to $\tau'$.
\end{proof}

\section{Mock-finite fields}
Recall that $\Abs(K)$ denotes the substructure of \emph{absolute
  numbers} of $K$---the elements algebraic over the prime field.
\begin{definition}
  A field $K$ is \emph{mock-finite} if $K$ is pseudofinite and
  $\Abs(K)$ is finite.
\end{definition}
We will see that mock-finite fields admit particularly nice Euler
characteristics.
\begin{definition}A field $K$ is a \emph{mock-$\Ff_q$} if $K$ is pseudofinite
    and $\Abs(K) \cong \Ff_q$.
\end{definition}
Note that $K$ is mock-finite if and only if $K$ is a mock-$\Ff_q$ for some $q$.
\begin{lemma}
  Let $K$ be a mock-$\Ff_q$.  Then the restriction homomorphism
  \[ \Gal(K) \to \Gal(\Ff_q)\]
  is an isomorphism.  Consequently, there is a unique topological
  generator $\sigma \in \Gal(K)$ extending the $q$th power Frobenius
  $\phi_q \in \Gal(\Ff_q)$.
\end{lemma}
\begin{proof}
  The restriction homomorphism is surjective because $\Ff_q$ is
  relatively algebraically closed in $K$.  Both Galois groups are
  isomorphic to $\hat{\Zz}$, and any continuous surjective
  homomorphism $\hat{\Zz} \to \hat{\Zz}$ is an isomorphism.
\end{proof}
\begin{definition}
  If $K$ is a mock-$\Ff_q$, the \emph{mock Frobenius automorphism} is
  the unique $\sigma \in \Gal(K)$ extending the $q$th-power Frobenius
  $\phi_q \in \Gal(\Ff_q)$.
\end{definition}
If $p$ is a prime, let $\Zz_{\neg p}$ be the prime-to-$p$ completion
of $\Zz$:
\begin{equation*}
  \Zz_{\neg p} = \varprojlim_{(n,p) = 1} \Zz/n\Zz = \prod_{\ell \ne p}
  \Zz_\ell.
\end{equation*}
\begin{definition}
  Let $K$ be a mock-finite field, and $\sigma$ be the mock Frobenius
  automorphism.
  \begin{enumerate}
  \item The \emph{principal Euler characteristic} on $K$ is the
    $\Zz_{\neg p}$-valued Euler characteristic induced by $\sigma$.
  \item The \emph{dual Euler characteristic} on $K$ is the $\Zz_{\neg
    p}$-valued Euler characteristic induced by $\sigma^{-1}$.
  \end{enumerate}
\end{definition}
The reason for the prime-to-$p$ restriction will become clear soon.
\begin{lemma}
  The principal and dual Euler characteristics are 0-definable.
\end{lemma}
\begin{proof}
  They are definable by Lemma~\ref{psf-case}, and
  $\Aut(K/\emptyset)$-invariant by construction.
\end{proof}

\subsection{Mock-frobenius periodic fields}
\begin{definition}
  A periodic field $(K,\sigma)$ is a \emph{mock-$\Fr^q$} if
  $(K,\sigma) \models \ACPF$ and $\Abs(K,\sigma) \cong \Fr^q$.
\end{definition}
\begin{proposition}\label{prop:frob}
  Let $q$ be a prime power.
  \begin{enumerate}
  \item \label{z1} The theory of mock-$\Fr^q$ periodic fields is consistent and
    complete.
  \item \label{z2} If $K$ is a mock-$\Ff_q$ and $\sigma$ is the mock Frobenius,
    then $(K^{alg},\sigma)$ is a mock-$\Fr^q$.  Every mock-$\Fr^q$
    arises this way.
  \item \label{z3} The theory of mock-$\Ff_q$ fields is consistent and
    complete.
  \end{enumerate}
\end{proposition}
\begin{proof}
  \begin{enumerate}
  \item Mock-$\Fr^q$ fields exist because we can embed $\Fr^q$ into an
    existentially closed periodic field.  Any two mock-$\Fr^q$ fields
    are elementarily equivalent by Lemma~\ref{ee-criterion}.
  \item Clear from Proposition~\ref{acpf-is-pseuf} and the
    definitions.
  \item Combine \ref{z1} and \ref{z2}. \qedhere
  \end{enumerate}
\end{proof}

\subsection{The principal Euler characteristic}
Dwork proved the following part of the Weil conjectures, in
\cite{dwork}.
\begin{fact}\label{dw-fact}
  If $V$ is a variety over $\Ff_q$, then there are non-zero algebraic
  integers $\alpha_1, \ldots, \alpha_m$ and $\beta_1, \ldots,
  \beta_{\ell}$ such that for every $n$,
  \[|V(\Ff_{q^n})| = \alpha_1^n + \cdots + \alpha_m^n - \beta_1^n - \cdots - \beta_{\ell}^n.\]
\end{fact}
There is no assumption that $V$ is smooth, proper, or connected.
\begin{lemma}\label{dwork-case}
  Let $V, \alpha_i, \beta_j$ be as in Fact~\ref{dw-fact}.  Let
  $(K_\infty,\sigma)$ be a mock $\Fr^q$, and let $\chi_\ell$ be the
  $\ell$-adic part of the canonical Euler characteristic on $K$.  Then
  \begin{equation*}
    \chi_\ell(V(K_1)) = \alpha'_1 + \cdots + \alpha'_m - \beta'_1 -
    \cdots - \beta'_\ell, \label{to-check-elk}
  \end{equation*}
  where
  \begin{align*}
    \alpha'_i &= 
    \begin{cases}
      \alpha_i & v_\ell(\alpha_i) = 0 \\
      0 & v_\ell(\alpha_i) > 0
    \end{cases} \\
    \beta'_i &= 
    \begin{cases}
      \beta_i & v_\ell(\beta_i) = 0 \\
      0 & v_\ell(\beta_i) > 0
    \end{cases}
  \end{align*}
\end{lemma}
In other words, $\chi_\ell(V(K_1))$ is obtained from $|V(\Ff_q)|$ by
dropping the terms of positive $\ell$-adic valuation.
\begin{proof}
  Take a non-principal ultrafilter $\mathcal{U}$ on $\Nn$,
  concentrating on the sets $1 + n\Zz$ for every non-zero ideal
  $n\Zz$.
    \begin{claim}
    The ultralimit of $|V(\Ff_{q^n})|$ in $\Zz_\ell$ is given by the
    right-hand side of (\ref{to-check-elk}).
  \end{claim}
  \begin{claimproof}
    Take a finite extension $L/\Qq_\ell$ such that $L$ contains all
    the $\alpha_i$ and $\beta_j$.  Let $\Oo$ be the $\ell$-adic
    valuation ring on $L$.  It suffices to show that for any non-zero
    ideal $I \lhd \Oo$, the following is true for ``most'' $n$:
    \begin{equation}
      \alpha_i^n \equiv \alpha_i' \pmod{I} \label{to-check}
    \end{equation}
    If $v_\ell(\alpha_i) > 0$, then $\lim_{n \to \infty} \alpha_i^n =
    0$, and Equation (\ref{to-check}) holds because $\mathcal{U}$ is
    non-principal.  If $v_\ell(\alpha_i) = 0$, then $\alpha_i$ is a
    unit in the finite ring $\Oo/I$.  Let $m$ be the cardinality of
    the group of units $(\Oo/I)^\times$.  Then
    \begin{equation*}
      n \in m\Zz \implies \alpha_i^n \equiv 1 \pmod{I}.
    \end{equation*}
    Therefore
    \begin{equation*}
      n \in 1 + m\Zz \implies \alpha_i^n \equiv \alpha_i \pmod{I}.
    \end{equation*}
    By choice of $\mathcal{U}$, this holds for ``most'' $n$, verifying
    Equation (\ref{to-check}).
  \end{claimproof}
  \begin{claim}
    The ultralimit of $|V(\Ff_{q^n})|$ in $\Zz_\ell$ is given by the
    left-hand side of (\ref{to-check-elk}).
  \end{claim}
  \begin{claimproof}
    Let $(K'_\infty,\sigma')$ be the ultraproduct of Frobenius periodic fields
    \begin{equation*}
      (K'_\infty,\sigma') = \prod_{n \in \Nn} (\Ff_q^{alg},\phi_{q^n}) / \mathcal{U} .
    \end{equation*}
    Then $(K'_\infty,\sigma') \models \ACPF$ by
    Corollary~\ref{axish-1}.  For fixed $m$, note that
    \begin{equation*}
      n \equiv 1 \pmod{m} \iff \phi_{q^n} | \Ff_{q^m} = \phi_q | \Ff_{q^m}.
    \end{equation*}
    By choice of $\mathcal{U}$, it follows that $\sigma' | \Ff_{q^m} =
    \phi_q | \Ff_{q^m}$.  As this holds for all $m$,
    \begin{equation*}
      \Abs(K'_\infty,\sigma') \cong (\Ff_q^{alg},\phi_q).
    \end{equation*}
    So $(K'_\infty,\sigma')$ is a mock-$\Fr^q$, and
    $(K'_\infty,\sigma') \equiv (K_\infty,\sigma)$.  By uniform
    definability of $\chi_\ell$,
    \[ \chi_\ell(V(K'_1)) = \chi_\ell(V(K_1)).\]
    The canonical Euler characteristic on $K'_\infty$ is given by
    nonstandard counting, so $\chi_\ell(V(K'_1))$ is the ultralimit
    of $|V(\Fr^{q^n}_1)| = |V(\Ff_{q^n})|$.
  \end{claimproof}
  Now combine the two claims.
\end{proof}

\begin{lemma}\label{curve-case}
  Let $(K_\infty,\sigma)$ be a mock-$\Fr^q$.
  \begin{enumerate}
  \item \label{cc1} If $C$ is a curve over $\Ff_q$, and $\alpha_1,
    \ldots, \alpha_{2g}$ are the characteristic roots of the
    $q$th-power Frobenius, then the prime-to-$p$ part of
    $\chi(C(K_1))$ equals $|C(\Ff_q)|$.
  \item \label{cc2} If $V$ is a 1-dimensional variety over $\Ff_q$,
    then the prime-to-$p$ part of $\chi(V(K_1))$ equals $|V(\Ff_q)|$.
  \item \label{cc3} If $X$ is an $\Fr^q$-definable subset of $K_1$, then
    the prime-to-$p$ part of $\chi(X)$ equals $|X \cap \Ff_q|$.
  \end{enumerate}
\end{lemma}
\begin{proof}
  \hfill
  \begin{enumerate}
  \item By the Weil conjectures for curves (\cite{hartshorne},
    Appendix C, \S 1), we know that
    \[ C(\Ff_{q^n}) = 1 - \alpha_1^n - \cdots - \alpha_{2g}^n + q^n\]
    for all $n$.  Moreover, the Poincare duality part of the Weil
    conjectures gives an equality of multi-sets:
    \[ \{\alpha_1, \ldots, \alpha_{2g}\} = \{q/\alpha_1, \ldots, q/\alpha_{2g}\}\]
    It follows that each $\alpha_i$ is a unit in $\Qq_\ell^\times$,
    for $\ell$ prime to $p$.  Therefore, by Lemma~\ref{dwork-case},
    the $\ell$-adic part of $\chi(C(K))$ is given by
    \begin{equation*}
      \chi(C(K)) = 1 - \alpha_1 - \cdots - \alpha_{2g} + q =
      |C(\Ff_q)|.
    \end{equation*}
  \item Similar to Lemma~\ref{qfree-to-curves-v2}, one can produce an
    $\Ff_q$-variety $C$ and open subvarieties $V' \subseteq V$ and $C'
    \subseteq C$ such that
    \begin{itemize}
    \item $V'$ is isomorphic to $C'$ (over $\Ff_q$)
    \item The complements $V \setminus V'$ and $C \setminus C'$ have
      finitely many $K_1$-points.
    \item When base changed to $\Ff_q^{alg}$, $C$ is a finite disjoint
      union of curves.
    \end{itemize}
    Every $K^{alg}$-point of $V'\setminus V$ and $C' \setminus C$ is
    in $\Ff_q^{alg}$.  Therefore
    \begin{align*}
      \chi(V(K_1)) - |V(\Ff_q)| &= \chi(V'(K_1)) - |V'(\Ff_q)| =
      \chi(C'(K_1)) - |C'(\Ff_q)| \\ & = \chi(C(K_1)) - |C(\Ff_q)| = 0,
    \end{align*}
    where the final equality is Part~\ref{cc1}.
  \item By Proposition~\ref{all-in-zero}, there is a quasi-finite
    morphism $V \to \Aa^1_{\Ff_q}$ of $\Ff_q$-varieties such that $X$
    is the image of $V(K_1) \to \Aa^1(K_1) = K_1$.  For each $n$, let
    $V_n$ be the fiber product of $n$ copies of $V$ over $\Aa^1$.
    Then $V_n \to \Aa^1_{\Ff_q}$ is still quasi-finite, so $V_n$ has
    dimension at most 1.  Let $Y_n$ be the definable set $V_n(K_1)$.
    By Part~\ref{cc2}, $\chi(Y_n) = |Y_n(\Fr^q)|$.

    Now use the argument of Lemma~\ref{almost-unique-lemma-1}.  Let $f
    : Y_1 \to X$ be the surjection induced by $V \to \Aa^1$.  Let
    $X_k$ be the definable set of $a \in X$ such that the fiber
    $f^{-1}(a)$ has size $m$.  Note that if $a \in X(\Fr^q)$, then every
    point in the fiber is field-theoretically algebraic over $a$,
    hence in $Y_1(\Fr^q)$.

    The upshot is that the fibers of $Y_1(\Fr^q) \to X(\Fr^q)$ have size
    $k$ over $X_k(\Fr^q)$, and more generally the fibers of $Y_n(\Fr^q)
    \to X(\Fr^q)$ have size $k^n$ over $X_k(\Fr^q)$.  Therefore,
    \begin{align*}
      |Y_n(\Fr^q)| &= \sum_k k^n \cdot |X_k(\Fr^q)| \\
      \chi(Y_n) &= \sum_k k^n \cdot \chi(X_k),
    \end{align*}
    where the second line is as in the proof of
    Lemma~\ref{almost-unique-lemma-1}.  By Part~\ref{cc2}, the left
    hand sides agree.  By the invertibility of Vandermonde matrices,
    it follows that $\chi(X_k) = |X_k(\Fr^q)|$.  Summing over $k$, we
    see $\chi(X) = |X(\Fr^q)| = |X \cap \Ff_q|$. \qedhere
  \end{enumerate}
\end{proof}

\begin{proposition}
  Let $K$ be a mock-$\Ff_q$.  Let $\chi$ be the principal Euler
  characterisic on $K$.  For any definable set $X \subseteq K^n$, we
  have
  \[ \chi(X) = |X \cap \Ff_q^n|.\]
  In particular, $\chi(X) \in \Zz$.
\end{proposition}
\begin{proof}
  Proceed by induction on $n$.  For the base case $n = 1$, expand $K$
  to a mock-$\Fr^q$ by Proposition~\ref{prop:frob}.\ref{z2}, and then
  apply Lemma~\ref{curve-case}.\ref{cc3}.  Suppose $n > 1$.  For $a
  \in K_1$, let $X_a$ denote the slice of $X$ over $a$:
  \begin{equation*}
    X_a = \{ \vec{b} \in (K_1)^{n-1} : (a,\vec{b}) \in X\}.
  \end{equation*}
  Fix $\ell^k$, and work with $\chi$ modulo $\ell^k$.  For $i \in
  \Zz/\ell^k$, let $S_i$ be the set of $a \in K_1$ such that
  $\chi(X_a) \equiv i \pmod{\ell^k}$.  Each set $S_i$ is
  $\Fr^q$-definable, so by induction $\chi(S_i) = |S_i \cap \Ff_q|$.
  Now for $a \in S_i \cap \Ff_q$, the set $X_a$ is $\Fr^q$-definable, so
  by induction $\chi(X_a) = |X_a \cap \Ff_q^{n-1}|$.  Then the
  following holds modulo $\ell^k$:
  \begin{align*}
    \chi(X) &\equiv \sum_{i \in \Zz/\ell^k} i \cdot \chi(S_i) \equiv \sum_{i \in \Zz/\ell^k} i \cdot |S_i \cap \Ff_q| \\
    & \equiv \sum_{i \in \Zz/\ell^k} \sum_{a \in S_i \cap \Ff_q} i \equiv \sum_{i \in \Zz/\ell^k} \sum_{a \in S_i \cap \Ff_q} \chi(X_a) \\
    & \equiv \sum_{a \in \Ff_q} \chi(X_a) \equiv \sum_{a \in \Ff_q} |X_a \cap \Ff_q^{n-1}|.
  \end{align*}
  The final sum is $|X \cap \Ff_q^n|$.
\end{proof}
This lets us simplify Lemma~\ref{dwork-case}:
\begin{corollary}\label{re-dwork}
  Let $V, \alpha_i, \beta_j$ be as in Fact~\ref{dw-fact}.  Let $F$ be
  a mock $\Ff_q$, and $\chi$ be its principal Euler characteristic.  Then
  \begin{equation*}
    \chi(V(F)) = \alpha_1 + \cdots + \alpha_m - \beta_1 - \cdots - \beta_\ell.
  \end{equation*}
\end{corollary}
This implies something about the numbers appearing in Dwork's theorem.
\begin{corollary}\label{basically-units}
  If $V$ is a variety over $\Ff_q$, then the $\alpha_i$ and $\beta_j$
  of Fact~\ref{dw-fact} have $\ell$-adic valuation zero for $\ell$
  prime to $q$.
\end{corollary}
\begin{proof}
  Let $\alpha'_i$ and $\beta'_i$ be as in Lemma~\ref{dwork-case}.  Let
  $F$ be a mock-$\Ff_q$, and $\chi_\ell$ be the $\ell$-adic part of
  the principal Euler characteristic.  Comparing
  Lemma~\ref{dwork-case} and Corollary~\ref{re-dwork}, we see
  \begin{equation*}
    \alpha'_1 + \cdots + \alpha'_m - \beta'_1 - \cdots - \beta'_\ell =
    \alpha_1 + \cdots + \alpha_m - \beta_1 - \cdots - \beta_\ell.
  \end{equation*}
  Replacing $\Ff_q$ with $\Ff_{q^n}$ changes $\alpha_i$ to
  $\alpha_i^n$ and $\alpha'_i$ to $(\alpha'_i)^n$.  Therefore, the
  following holds for any $n \ge 1$:
  \begin{equation*}
    (\alpha'_1)^n + \cdots + (\alpha'_m)^n - (\beta'_1)^n - \cdots -
    (\beta'_\ell)^n = \alpha_1^n + \cdots + \alpha_m^n - \beta_1^n -
    \cdots - \beta_\ell^n.
  \end{equation*}
  Comparing Poincare series, one gets equality of multisets
  \begin{align*}
    \{\alpha'_1,\ldots,\alpha'_m\} &= \{\alpha_1,\ldots,\alpha_m\} \\
    \{\beta'_1,\ldots,\beta'_\ell\} &= \{\beta_1,\ldots,\beta_\ell\}.
  \end{align*}
  Therefore, none of the $\alpha'_i$ or $\beta'_i$ are zero, and every
  $\alpha_i$ and $\beta_i$ has $\ell$-adic valuation 0.
\end{proof}
\begin{remark}
  Corollary~\ref{basically-units} can be proven using $\ell$-adic
  cohomology, but the proof given here is more elementary.
\end{remark}
\subsection{The dual Euler characteristic} \label{sec:anti}
Let $K$ be a mock-$\Ff_q$.  Recall that the \emph{dual Euler
  characteristic} on $K$ is the prime-to-$q$ part of the canonical
Euler characteristic induced by $\sigma^{-1}$, where $\sigma$ is the
mock Frobenius.
\begin{lemma}\label{anti-dwork}
  Let $V$ be a variety over $\Ff_q$, and let $\alpha_1, \ldots,
  \alpha_m, \beta_1, \ldots, \beta_\ell$ be the algebraic integers
  from Fact~\ref{dw-fact}.  Let $K$ be a mock-$\Ff_q$ and let
  $\chi^\dag$ be the dual Euler characteristic.  Then
  \begin{equation*}
    \chi^\dag(V(K)) = \alpha_1^{-1} + \cdots + \alpha_m^{-1} - \beta_1^{-1}
    - \cdots - \beta_\ell^{-1}.
  \end{equation*}
  Moreover, this value is rational.
\end{lemma}
\begin{proof}
  Similar to Lemma~\ref{dwork-case}, but using an ultrafilter that
  concentrates on $-1 + n\Zz$ for all $n$.
  Corollary~\ref{basically-units} ensures that $v_\ell(\alpha_i) = 0$
  for all $i$, so there is no need for any $\alpha'_i$'s or
  $\beta'_i$'s.  Rationality is an easy exercise, using the fact that
  \[ \alpha_1^n + \cdots + \alpha_m^n - \beta_1^n - \cdots - \beta_\ell^n \in \Zz\]
  for all $n \in \Nn$.
\end{proof}
\begin{proposition}\label{rational}
  If $K$ is a mock-$\Ff_q$ and $\chi^\dag$ is the dual Euler
  characteristic on $K$, then $\chi^\dag(X) \in \Qq$ for every
  $\Ff_q$-definable set $X$.
\end{proposition}
\begin{proof}
  If $X$ is the set of $K$-points in some $\Ff_q$-definable variety,
  this follows from Lemma~\ref{anti-dwork}.

  If $X$ is a definable subset of $K^n$, then
  Proposition~\ref{all-in-zero} yields a quasi-finite morphism $V \to
  \Aa^n$ of varieties over $\Ff_q$, such that $X$ is the image of
  $V(K_1) \to \Aa^n(K_1)$.  Let $V_n$ be the $n$-fold fiber product of
  $V$ over $\Aa^1$.  By the argument of
  Lemma~\ref{almost-unique-lemma-1}, $\chi^\dag(X)$ is given by some
  rational linear combination of the $\chi^\dag(V_n(K))$.
\end{proof}
\begin{example}\label{smooth-example}
  If $V$ is a $d$-dimensional smooth projective variety over $\Ff_q$,
  then the following identities of multisets hold by the Poincare
  duality part of the Weil conjectures:
  \begin{align*}
    \{\alpha_1,\ldots,\alpha_m\} &= \{q^d/\alpha_1,\ldots,q^d/\alpha_m\} \\
    \{\beta_1,\ldots,\beta_{m'}\} &= \{q^d/\beta_1,\ldots,q^d/\beta_{m'}\}.
  \end{align*}
  Therefore, for $K$ a mock-$\Ff_q$ with dual Euler characteristic $\chi^\dag$,
  \begin{align*}
    \chi^\dag(V(K)) &= \alpha_1^{-1} + \cdots + \alpha_m^{-1} - \beta_1^{-1} - \cdots - \beta_{m'}^{-1} \\
    &= (\alpha_1 + \cdots + \alpha_m - \beta_1 - \cdots - \beta_{m'})/q^d
    = |V(\Ff_q)|/q^d.
  \end{align*}
\end{example}
Putting everything together, we have proven:
\begin{theorem-star}[Theorem~\ref{mocktastic}]
  Let $K$ be a mock-$\Ff_q$, for some prime power $q = p^k$.  There
  are two $\Zz_{\neg p}$-valued 0-definable strong Euler
  characteristics $\chi$ and $\chi^\dag$ on $K$, such that
  \begin{enumerate}
  \item If $V$ is a smooth projective variety over $\Ff_q$, then
    \begin{align*}
      \chi(V(K)) &= |V(\Ff_q)| \\
      \chi^\dag(V(K)) &= |V(\Ff_q)|/q^{\dim V}.
    \end{align*}
  \item If $X$ is any $\Ff_q$-definable set, then
    \[ \chi(X) = |X \cap \dcl(\Ff_q)|.\]
    In particular, $\chi(X) \in \Zz$.
  \item If $X$ is any $\Ff_q$-definable set, then $\chi^\dag(X) \in \Qq$.
  \end{enumerate}
\end{theorem-star}

\subsection{The neutral Euler characteristic}\label{sec-neutral}
Using the dual Euler characteristic on mock-finite fields, one can
produce an exotic $\Qq$-valued Euler characteristic $\chi_0$ on any
pseudofinite field of characteristic 0. We outline the construction,
omitting details because $\chi_0$ is less interesting than first
expected (see Example~\ref{not-strong}).
\begin{enumerate}
\item Let $F$ be a pseudofinite field of characteristic zero, given
  explicitly as an ultraproduct of prime fields \[ F =
  \prod_i F_i / \mathcal{U} .\] Suppose none of the $F_i$ have
  characteristic $\ell$.
\item For each $i$, let $K_i$ be a mock-$F_i$.  Let $K$ be the
  ultraproduct \[ F = \prod_i K_i /\mathcal{U}.\] One can show that
  $F \preceq K$.
\item Each $K_i$ has its dual Euler characteristic
  $\chi^\dag_i$, taking values in the ring $\Zz_\ell$.  On
  $F_i$-definable sets, the Euler characteristic takes values in
  $\Zz_\ell \cap \Qq$.
\item Let $\Zz_\ell^*$ and $\Qq^*$ denote the ultrapowers
  $\Zz_\ell^{\mathcal{U}}$ and $\Qq^{\mathcal{U}}$.  In the
  nonstandard limit, the $\chi^\dag_i$ determine a
  $\Zz_\ell^*$-valued Euler characteristic $\tilde{\chi}$ on $K$.
  When restricted to $F$-definable sets, $\tilde{\chi}$ takes values
  in $\Qq^*$.  Because $F \preceq K$, this gives a $\Qq^*$-valued
  Euler characteristic $\tilde{\chi}$ on $F$.
\item Say that a weak Euler characteristic $\chi$ is \emph{medial} if
  it satisfies the following partial version of strongness
  \begin{quote}
    If $f : X \to Y$ is a definable surjection between two definable
    sets, and every fiber has size $k < \infty$, then $\chi(X) = k
    \cdot \chi(Y)$.
  \end{quote}
  One can verify that $\tilde{\chi}$ is a $\Qq^*$-valued medial Euler
  characteristic on $F$.
\item If $V$ is a geometrically irreducible smooth projective variety
  over $F$, one can show using Example~\ref{smooth-example} and the
  Lang-Weil estimates that $\tilde{\chi}(V(F))$ is infinitesimally
  close to 1.
\item Using resolution of singularities and induction on dimension,
  one can show that for any variety $V/F$, the value
  $\tilde{\chi}(V(F))$ is infinitesimally close to an integer.
\item Using Proposition~\ref{all-in-zero} and an argument similar to
  Lemma~\ref{almost-unique-lemma-1}, one can show that if $X \subseteq
  F^n$ is definable, then $\tilde{\chi}(X)$ is infinitesimally close
  to a rational number.
\item Define $\chi_0(X)$ to be the standard part of $\tilde{\chi}(X)$.
  Then $\chi_0$ is a $\Qq$-valued medial Euler characteristic on $F$.
  Also, $\chi_0(V(F)) = 1$ for any geometrically irreducible smooth
  projective variety $V/F$.
\item If $F'$ is \emph{any} pseudofinite field of characteristic 0,
  then there is at most one $\Qq$-valued medial Euler characteristic
  $\chi_0$ such that $\chi_0(V(F')) = 1$ for any geometrically
  irreducible smooth projective variety $V/F'$.  This can be seen by
  resolution of singularities, induction on dimension,
  Proposition~\ref{all-in-zero}, and the proof of
  Lemma~\ref{almost-unique-lemma-1}.
\item Therefore, every pseudofinite field of characteristic 0 admits
  a unique 0-definable $\Qq$-valued medial Euler characteristic
  $\chi_0$ characterized by the fact that $\chi_0(V(F)) = 1$ for any
  smooth projective geometrically connected variety $V$.  This follows
  by a Beth implicit definability argument, similar to the proof of
  Proposition~\ref{exactly-one} and Theorem~\ref{ecpdf-thm}.
\end{enumerate}
We call $\chi_0$ the \emph{neutral} Euler characteristic.
\begin{remark}
  Unlike the Euler characteristic of Lemma~\ref{psf-case}, $\chi_0$ is
  completely canonical, and is independent of the choice of a
  nonstandard Frobenius.  If resolution of singularities holds in
  positive characteristic, then $\chi_0$ can be defined for all
  pseudofinite fields.
\end{remark}
Unfortunately, the neutral Euler characteristic has bad properties:
\begin{example}\label{not-strong}
  The neutral Euler characteristic is \emph{not} strong.  Consider the
  set
  \begin{equation*}
    S = \{(x,y,\lambda) : y^2 = x(x-1)(x-\lambda) \text{ and } \lambda \ne 0, 1\}.
  \end{equation*}
  One can view $S$ as a family of elliptic curves $E_\lambda$
  parameterized by $\lambda$.  Each elliptic curve $E_\lambda$ is
  missing one point at infinity, so
  \[ \chi_0(E_\lambda) = \chi_0(\overline{E}_\lambda) - 1 = 1 - 1 = 0.\]
  Therefore, if $\chi_0$ were a strong Euler characteristic, one would have
  \begin{equation*}
    \chi_0(S) = 0 \cdot \chi_0(\Pp^1(F) \setminus \{0,1,\infty\}) = 0 \cdot (1-3) = 0.
  \end{equation*}
  On the other hand, one can directly count points in $S$.  For any
  finite field $\Ff_q$ of characteristic $\ne 2$, the size of
  $S(\Ff_q)$ turns out to be given by the formula
  \begin{equation*}
    |S(\Ff_q)| = 
    \begin{cases}
      q^2 - q & \text{ if $-1$ is a square in $\Ff_q$} \\
      q^2 - q - 2 & \text{ if $-1$ is not a square in $\Ff_q$}.
    \end{cases}
  \end{equation*}
  Essentially by Lemma~\ref{anti-dwork}, one sees that the dual Euler
  characteristic of $S$ is given by
  \begin{equation*}
    \begin{cases}
      q^{-2} - q^{-1} & \text{ if $-1$ is a square in $\Ff_q$} \\
      q^{-2} - q^{-1} - 2 & \text{ if $-1$ is not a square in $\Ff_q$}.
    \end{cases}
  \end{equation*}
  In the limit, $q^{-1}, q^{-2} \to 0$.  Consequently, the neutral
  Euler characteristic of $S(F)$ is given by
  \begin{equation*}
    \chi_0(S(F)) = 
    \begin{cases}
      0 & \text{ if $-1$ is a square in $F$} \\
      -2 & \text{ if $-1$ is not a square in $F$}.
    \end{cases}
  \end{equation*}
  In particular, $\chi_0(S(F))$ need not equal 0.
\end{example}
The neutral Euler characteristic seems to be governed by the weight 0
part of $\ell$-adic etale cohomology, in a manner analogous to
Conjecture~\ref{weilish}.

\section{Directions for future research}
There are several immediate directions for future research.  The most
important next step is verifying Conjecture~\ref{horror}, completing
the proof that the $\mathcal{L}^{\mu}_{rings}$-theory of finite fields
is decidable (Theorem~\ref{intro-deci}).  This will hopefully be carried
out in \cite{complementary-paper}.  Another key task is to relate the
$\Zz_\ell$-valued Euler characteristic to $\ell$-adic etale cohomology
(Conjecture~\ref{weilish}).

Another interesting direction is the following variant of
Theorem~\ref{intro-deci}:
\begin{conjecture}
  The $\mathcal{L}^{\mu}_{rings}$-theory of the rings $\Zz/n\Zz$ is
  decidable.
\end{conjecture}
The conjecture can be broken into several cases:
\begin{enumerate}
\item The rings $\Zz/p\Zz = \Ff_p$.  These are essentially
  handled by Theorem~\ref{intro-deci}.
\item The rings $\Zz/p^k\Zz$.  For fixed $p$, the theory of these
  rings is closely related to $p$-adically closed fields.  Indeed, any
  ultraproduct of $\Zz/p^k\Zz$ is interpretable in a $p$-adically
  closed field.  If $p$ is allowed to vary, one also encounters
  henselian valued fields with pseudofinite residue field of
  characteristic 0.  These theories are well-understood, and admit a
  form of cell decomposition.  Mimicking the proofs of motivic
  integration, it should be possible to verify that elimination of
  parity quantifiers holds in the rings $\Zz/p^k\Zz$---modulo naming
  the nonstandard Frobenius.
\item The rings $\Zz/n\Zz$, where $n$ has multiple prime divisors.
  The theory of such rings can be analyzed using Feferman-Vaught
  theory, mimicking \cite{adeles}.  Elimination of parity quantifiers
  fails rather strongly, but can be recovered by expanding the boolean
  algebra sort with new predicates.
\end{enumerate}
Lastly, it may be possible to generalize the definability of the
canonical Euler characteristic from ACPF to its expansion ACFA.
Although ACFA is not pseudofinite, its models are ultraproducts of
Frobenius difference fields (\cite{frobenius}), and definable sets of
finite rank are naturally pseudofinite.

\subsection{Interactions with number theory?}
We have relied heavily on algebraic geometry and number theory to
prove a relatively simple model-theoretic fact.  One could dream of
reversing the process to obtain new results in number theory.
Ultraproducts of finite fields are not the only source of pseudofinite
fields.  For example, if $\sigma$ is chosen randomly in $\Gal(\Qq)$,
then $(\Qq^{alg},\sigma) \models \ACPF$ with probability 1, by
(\cite{field-arithmetic}, \S16.6).  Perhaps one can prove non-trivial
facts by reasoning about nonstandard sizes of definable sets in these
structures.

Unfortunately, we have probably done nothing interesting from a
number-theoretic point of view.  The nonstandard ``sizes'' on
pseudofinite fields should be a simple artifact of etale cohomology,
by Conjecture~\ref{weilish}.  Etale cohomology is already
well-understood.  Combinatorial facts about sizes correspond to
well-known facts about cohomology.  The fact that $\chi(X \times Y) =
\chi(X) \cdot \chi(Y)$ corresponds to the K\"unneth formula.  When $f
: E \to B$ is a morphism, the strong Euler characteristic property
allows us to calculate the ``size'' of $E$ by ``integrating'' the
``sizes'' of the fibers over $B$.  This property corresponds to the
Leray spectral sequence.

One tool which might be new on the model-theoretic side is elimination
of imaginaries, which holds in ACPF by work of
Hrushovski~\cite{the-lost-paper}.  When $X$ is interpretable, or
definable with quantifiers, we know that $\chi(X)$ is ``integral,''
lying in $\hat{\Zz}$ rather than $\hat{\Zz} \otimes_\Zz \Qq$.  There
may be some number-theoretic content to this.

It feels as if there could be some connection between the canonical
Euler characteristic and $p$-adic L-functions.  The classical
L-functions associated to number fields and elliptic curves are
defined in terms of point counting.  In some cases, these L-functions
can be converted to $p$-adic analytic functions by extrapolating the
values at negative integers.  Insofar as we are counting points on
varieties mod $p^k$, there is a spiritual connection to the $p$-adic
part of the canonical Euler characteristic.

Moreover, $p$-adic integration appears in both contexts.  If $\chi$ is
a strong $\Zz_p$-valued Euler characteristic, and $f : E \to B$ is a
definable function, then $\chi$ induces a $p$-adic measure $\mu$ on $B$, and
one can calculate $\chi(E)$ by $p$-adic integration
\begin{equation*}
  \chi(E) = \int_{x \in B} \chi(f^{-1}(x))\, d \mu(x)
\end{equation*}
This was essentially how $\chi(X)$ was calculated in
Lemma~\ref{almost-unique-lemma-n}.  Meanwhile, $p$-adic integration
plays a key role in the theory of $p$-adic L-functions.  For example,
the Riemann zeta function is given on negative integers by a $p$-adic
Mellin transform: there is some $c \in \Zz_p^\times$ and $p$-adic
measure $\mu$ on $\Zz_p$ such that for positive integers $k$,
\begin{equation}
  \zeta(-k) = \frac{1}{1 - c^{k+1}} \int_{\Zz_p} x^k \, d\mu(x). \label{mellin}
\end{equation}
This Mellin transform is the underlying reason why the Kubota-Leopoldt
$p$-adic zeta function exists.  In some cases, the measure $\mu$ can
be given a pseudofinite interpretation.  For example, if $p$ is odd
and $\alpha$ is a nonstandard integer whose $p$-adic standard part is
$-1/2$, then $\zeta(-k)$ is given\footnote{Let $B_k(x)$ denote the $k$th Bernoulli
  polynomials
  \begin{equation*}
    \sum_{k = 0}^\infty \frac{B_k(x)t^k}{k!} = \frac{t e^{xt}}{e^t -
      1},
  \end{equation*}
  and let $B_k$ denote the Bernoulli numbers $B_k(0)$.  The identity
  \begin{equation*}
    \frac{B_{k+1}(1/2) - B_{k+1}(0)}{k+1} = (2^{-k}
    - 2)\frac{B_{k+1}}{k+1}
  \end{equation*}
  can be proven by an easy exercise in generating functions.  Let
  $\approx$ denote equality of standard parts.  Then the following
  holds for positive integers $k$,
  \begin{align*}
    \zeta(-k) &= \frac{-B_{k+1}}{k+1} = \frac{1}{2 -
      2^{-k}}\left(\frac{B_{k+1}(1/2) - B_{k+1}(0)}{k+1}\right) \\ &
    \approx \frac{1}{2 - 2^{-k}} \left(\frac{B_{k+1}(\alpha + 1) -
      B_{k+1}(0)}{k+1}\right) = \frac{1}{2 - 2^{-k}} \sum_{n =
      1}^\alpha n^k,
  \end{align*}
  by well-known relations between the Bernoulli numbers, the zeta
  function, and sums of $k$th powers.}
by $p$-adic standard part of the sum
\begin{equation*}
  \frac{1}{2 - 2^{-k}} \sum_{n = 1}^\alpha n^k.
\end{equation*}
In other words, (\ref{mellin}) holds with $c = 1/2$ and $\mu$ equal to
(half) the nonstandard counting measure on the pseudofinite set
$\{1,2,\ldots,\alpha\}$.

Thus there are several vague connections between the canonical Euler
characteristic on pseudofinite fields, and $p$-adic L-functions.  I
lack the expertise to pursue this connection further.

\bibliographystyle{plain}
\bibliography{mybib}{}

\end{document}